\documentclass[11pt]{amsart}

\usepackage{graphicx,amssymb,pstricks,subfigure}

\input xy
\xyoption{all}

\newtheorem{theorem}{Theorem}

\newtheorem{proposition}[theorem]{Proposition}

\theoremstyle{definition}

\def\R{\mathbb{R}}
\def\Z{\mathbb{Z}}
\def\std{{\operatorname{std}}}
\def\rot{{\operatorname{rot}}}
\def\NW{\text{\it X:NW}}
\def\NE{\text{\it X:NE}}
\def\SW{\text{\it X:SW}}
\def\SE{\text{\it X:SE}}
\def\HFKhat{\widehat{\textit{HFK}}}
\def\HFKtilde{\widetilde{\textit{HFK}}}
\def\CFKtilde{\widetilde{\textit{CFK}}}
\def\A{\mathfrak{A}}
\def\B{\mathfrak{B}}

\begin{document}

\title{A family of transversely nonsimple knots}
\author{Tirasan Khandhawit}
\author{Lenhard Ng}
\address{Mathematics Department, Duke
University, Durham, NC 27708}
\email{ng@math.duke.edu}
\urladdr{http://www.math.duke.edu/\~{}ng/}

\begin{abstract}
We apply knot Floer homology to exhibit an infinite family of
transversely nonsimple prime knots
starting with $10_{132}$. We also
discuss the combinatorial relationship between grid diagrams, braids, and
Legendrian and transverse
knots in standard contact $\mathbb{R}^3$.
\end{abstract}

\maketitle

%*********************************************************************
%*********************************************************************

\section{Introduction}

Transverse knots play an important role in contact topology, but
surprisingly little is known about their classification even in the
simplest setting, $\R^3$ with the standard contact structure. Any
transverse knot has an underlying topological knot type, and it also
carries at least one other piece of data, the self-linking number. A
topological knot type in $\R^3$ is \textit{transversely simple} if
transverse knots of that underlying type are completely classified by
their self-linking number; otherwise, it is \textit{transversely
  nonsimple}. Various knots are known to be transversely simple,
including the unknot \cite{bib:Eli}, torus knots \cite{bib:Et2}, and
the figure eight knot \cite{bib:EH}.

It was only recently that some knot types were shown to be
transversely nonsimple. Birman and Menasco \cite{bib:BM2} (see also
\cite{bib:BM1})
used braid theory to find a family of $3$-braids whose knot closures
are transversely nonsimple; Etnyre and Honda \cite{bib:EH2} used
contact-topological techniques to show that the $(2,3)$ cable of the
$(2,3)$ torus knot is transversely nonsimple.

There has been much current effort to develop invariants of transverse
knots that can be used to demonstrate transverse nonsimplicity. The
first (and thus far only) invariant that has been shown to be
effective lies, interestingly, not in contact-topological
constructions like Symplectic Field Theory, but in knot Floer
homology \cite{bib:OSz,bib:Ras}.
The $\widehat{\theta}$ invariant in $\HFKhat$ was introduced
by Ozsv\'ath, Szab\'o, and Thurston \cite{bib:OST} and was employed in
\cite{bib:NOT} to find several examples of transversely nonsimple
knots, including $10_{132}$ and (a reproof of) the Etnyre--Honda cable
example. V\'ertesi \cite{bib:Ver} used the examples of \cite{bib:NOT}
and the behavior of $\widehat{\theta}$ under connected sum to find
infinite families of connected-sum examples of transversely nonsimple
knots; see also \cite{bib:Kaw}.
Most recently, by studying the relationship between contact surgery and
naturality properties of (a differently constructed
version of) $\widehat{\theta}$, Ozsv\'ath and Stipsicz \cite{bib:OS} proved
transverse nonsimplicity for a wide family of two-bridge knots.

Absent the naturality techniques of \cite{bib:OS}, the applications of
$\widehat{\theta}$ to transverse simplicity have used a crude but
surprisingly effective ``vanishing criterion'':
if $T_1$ and $T_2$ are transverse
knots and $\widehat{\theta}(T_1) = 0$ while $\widehat{\theta}(T_2)
\neq 0$, then $T_1$ and $T_2$ are distinct. The $\widehat{\theta}$
invariant lies in the homology of a combinatorial chain complex
introduced in \cite{bib:MOS}, and \cite{bib:NOT} used a computer
program to determine in examples whether $\widehat{\theta}$ is
null-homologous or not. However, reliance on a computer program
obviously limits the number of transversely nonsimple examples that
can be found.

In this paper, we find a two-parameter infinite family of prime,
transversely nonsimple knots that can be distinguished using
the vanishing criterion for $\widehat{\theta}$. The idea is to find
grid diagrams (the structures on which $\widehat{\theta}$ is defined)
where the computation of $\widehat{\theta}$ is short enough to be
carried out by hand. The resulting family of examples is not as
simple in appearance as the two-bridge examples of \cite{bib:OS}, but
has the advantage of needing only the combinatorial description of
$\widehat{\theta}$ and not an analysis of its image under
contact surgery.

Our family is a generalization of the $10_{132}$ example from
\cite{bib:NOT}. A transverse knot can be represented as a braid (see
Section~\ref{ssec:braids}), and a fruitful technique for finding
transversely nonsimple knots is to find braids that are related by a
negative flype (cf.\ \cite{bib:BM2}) and thus represent the same
topological, but not necessarily transverse, knot. In correspondence
with the second author, H.\ Matsuda noted that the $10_{132}$ example
can be expressed as a negative flype, and proposed a one-parameter
family of braids generalizing $10_{132}$.
Here we expand Matsuda's conjectured family to a two-parameter family of braids
related by a negative flype.\footnote{We note in passing that the
  $7_2$ transverse knots in \cite[Figure 11]{bib:NOT}, shown to be
  distinct in \cite{bib:OS}, are also related by a negative flype,
  $\sigma_3^2\sigma_2^2\sigma_3^{-1}\sigma_1^2\sigma_2\sigma_1^{-1}
  \leftrightarrow
  \sigma_3^2\sigma_2^2\sigma_3^{-1}\sigma_1^{-1}\sigma_2\sigma_1^2$, as
  can be checked using the techniques from this paper.}
%Starting with an
%observation of H.\ Matsuda that the $10_{132}$ example can be
%expressed as a negative flype,
%we discovered the following result.

\begin{theorem}
For any $a,b\geq 0$, the pair of $4$-braids
\[
\sigma_3\sigma_2^{-2}\sigma_3^{2a+2}\sigma_2\sigma_3^{-1}
\sigma_1^{-1}\sigma_2\sigma_1^{2b+2}
\qquad \text{and} \qquad
\sigma_3\sigma_2^{-2}\sigma_3^{2a+2}\sigma_2\sigma_3^{-1}
\sigma_1^{2b+2}\sigma_2\sigma_1^{-1},
\]
related by a negative flype and thus representing the same topological
knot and having the same self-linking number, represent distinct
transverse knots. In particular, the topological knot types given by
these pairs, which are prime, are transversely nonsimple.
\label{thm:main}
\end{theorem}

\noindent
The $10_{132}$ case is $a=b=0$; in this case, the braids in Theorem~\ref{thm:main} are transversely isotopic to $L_1$ and $L_2$, respectively, from \cite[Section 3.1]{bib:OST}. Other small knots in this family, with
the corresponding values of $(a,b)$, include $(0,1) = 12^n_{120}$,
$(1,0) = 12^n_{199}$, $(0,2) = 14^n_{2016}$, $(1,1) = 14^n_{3606}$,
and $(2,0) = 14^n_{5045}$.

A.\ Shumakovitch has noticed that some
(perhaps all) of this family of knots have interesting odd Khovanov
homology \cite{bib:ORS}. More precisely, the six examples listed above
have the unusual feature that their odd Khovanov homology completely vanishes in homological degree $0$. We do not know if this is
a coincidence.

We believe that this two-parameter family is in fact part of a
four-parameter family of transversely nonsimple knots given by the
closures of the braids
\[
\sigma_3\sigma_2^{-2c-2}\sigma_3^{2a+2}\sigma_2\sigma_3^{-2d-1}
\sigma_1^{-1}\sigma_2\sigma_1^{2b+2}
\]
for $a,b,c,d\geq 0$. We have checked several examples using the
computer program of \cite{bib:NOT} but do not have a general proof for
the case $(c,d) \neq (0,0)$.

In order to apply the $\widehat{\theta}$ invariant to braids to prove
Theorem~\ref{thm:main}, we need techniques for translating
between braids, grid diagrams, and Legendrian and transverse knots in
standard contact $\R^3$. These translations are presented in
Section~\ref{sec:translate} and are by now well-known to experts, but
we were unable to find any full write-ups in the literature. In
particular, the algorithms for obtaining a Legendrian knot from a
braid and a braid from a grid diagram may be of independent
interest. We then prove Theorem~\ref{thm:main} in Section~\ref{sec:proof}.

%*********************************************************************
\subsection*{Acknowledgments}

We would like to thank H.\ Matsuda, A.\ Shumakovitch, and D.\ Thurston for
useful discussions. Part of this work also appeared in the first
author's undergraduate honors thesis at Duke University.
The second author is supported by NSF grant
DMS-0706777.

%*********************************************************************
%*********************************************************************

\section{Braids, Grid Diagrams, and Transverse Knots}
\label{sec:translate}

Here we review several different approaches to transverse knots in
standard contact $\R^3$. Most of the material in this section can be
found in the Etnyre survey \cite{bib:Et} or, in the case of grid
diagrams, \cite{bib:OST}. The new content consists of results in
Section~\ref{ssec:translate} giving methods to translate between grid
diagrams, braids, and Legendrian knots, but even these are ``folk
theorems'' that have been floating around the subject for several years.

%*********************************************************************
\subsection{Legendrian and transverse knots}
\label{ssec:legknots}

Let
\[
\xi_\std = \ker (dz-y\,dx)
\]
denote the standard contact structure
on $\R^3$. A \textit{Legendrian knot} in $(\R^3,\xi_\std)$ is an
oriented knot that is everywhere tangent to $\xi_\std$. A
\textit{transverse knot} is an oriented knot that is everywhere
transverse to $\xi_\std$, with the orientation agreeing with the usual
coorientation on $\xi_\std$; that is, $dz-y\,dx > 0$ along the
orientation of a transverse knot. Any smooth knot in $\R^3$ can be $C^0$
perturbed to both Legendrian and transverse knots. We consider
Legendrian and transverse knots up to \textit{Legendrian} and
\textit{transverse isotopy}, isotopy through Legendrian and transverse
knots, respectively.

There is a many-to-one correspondence between Legendrian and
transverse knots. Any Legendrian knot $L$ can be perturbed to a
transverse knot $L_+$, the \textit{positive pushoff} of
$L$, by pushing each point on $L$ in a direction transverse to the
contact plane; the positive pushoff is unique up to transverse
isotopy. There is an ``inverse'' operation that perturbs any
transverse knot to a Legendrian knot, but this is only well-defined up
to Legendrian isotopy and negative stabilization/destabilization of
Legendrian knots (see below). Thus one can view transverse knots up to
transverse isotopy as Legendrian knots up to Legendrian isotopy and
negative de/stabilization.

A convenient way to depict Legendrian and transverse knots is through
their \textit{front projections} to the $xz$ plane. The front
projection of a generic Legendrian knot has no vertical tangencies and
has only double points and semicubical cusps as singularities. At each
double point, the arc of more negative slope passes over the arc of
more positive slope. Any front of this type is the front projection of
a Legendrian knot.

\begin{figure}
\centerline{
\includegraphics[height=0.5in]{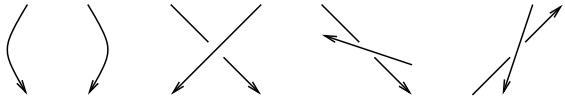}
}
\caption{Forbidden segments in the front projection of a transverse
  knot.}
\label{fig:forbidden}
\end{figure}

On the other hand, the front projection of a
generic transverse knot is a standard knot diagram, with only
double points as singularities, but with two restrictions: any point
in the projection with a vertical tangency must be oriented upwards,
and at a crossing, we cannot simultaneously have the overcrossing arc pointing
to the left, the undercrossing arc pointing to the right, and the
overcrossing arc of greater slope than the undercrossing arc. See
Figure~\ref{fig:forbidden}. Any knot projection without these
forbidden segments is the front projection of a transverse knot,
unique up to transverse isotopy, and two transverse knots are
transversely isotopic if and only if their front projections are
isotopic through diagrams that do not contain any forbidden segments.

\begin{figure}
\centerline{
\includegraphics[height=1.2in]{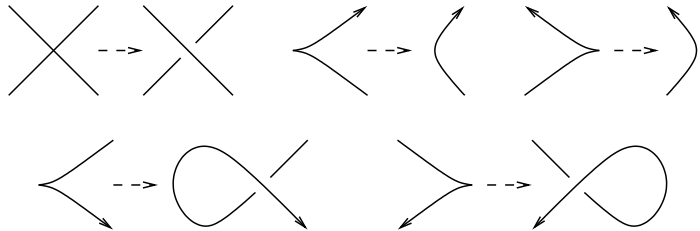}
}
\caption{
Obtaining the front of the transverse pushoff from a Legendrian front.
}
\label{fig:pushoff}
\end{figure}

With front projections, it is easy to see the correspondence between
Legendrian and transverse knots. The front of a Legendrian knot can be
turned into the front of its positive transverse pushoff by smoothing
out upward-pointing cusps and replacing downward-pointing cusps by
loops; see Figure~\ref{fig:pushoff}. In a related vein, we define the
\textit{positive}
and \textit{negative stabilizations} $S_{\pm}(L)$ of a Legendrian knot
$L$ to be the
Legendrian knots whose fronts are obtained from the front of $L$ by
adding in a zigzag whose cusps point downward or upward; see
Figure~\ref{fig:stab}. Both stabilizations are well-defined up to
Legendrian isotopy. It is clear from the front picture that a
Legendrian knot and its negative stabilization have positive
transverse pushoffs that are transversely isotopic. A result due in the
$\R^3$ case to \cite{bib:EFM} states that two Legendrian knots are
related by Legendrian isotopy and negative de/stabilization if and
only if their positive transverse pushoffs are related by transverse
isotopy.

\begin{figure}
\centerline{
\includegraphics[height=0.4in]{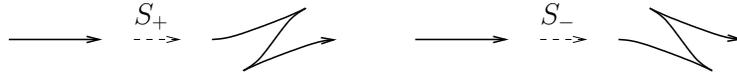}
}
\caption{
Positive and negative stabilizations of a Legendrian front.
}
\label{fig:stab}
\end{figure}

%*********************************************************************
\subsection{Grid diagrams}
\label{ssec:grid}

Closely related to front projections is another fruitful method of
representing Legendrian knots, via grid diagrams. A \textit{grid
  diagram} is an $n\times n$ square grid with a collection of $n$
$X$'s and $n$ $O$'s in the grid, such that each row or column
contains exactly one $X$ and one $O$, and no square in the grid
contains both an $X$ and an $O$.

One obtains a knot (or link) diagram from a grid diagram by connecting $O$'s
to $X$'s horizontally, connecting $X$'s to $O$'s vertically, and
stipulating that horizontal segments always pass over vertical
segments whenever they cross. (Note that this is the opposite of the
standard convention for grid diagrams.) In this way, any knot can be
represented by a grid diagram. Indeed, we can view a grid diagram $G$ as
(the front of) a Legendrian knot $L(G)$ by turning it $45^\circ$ clockwise,
smoothing upward- and downward-pointing corners, and turning leftward-
and rightward-pointing corners into cusps. See
Figure~\ref{fig:braidify} below. Any Legendrian knot is
Legendrian isotopic to a knot obtained in this way from some grid
diagram.

There is a sequence of elementary moves on grid diagrams, the Cromwell
moves \cite{bib:Cro}, that relate any two grid diagrams that represent
topologically isotopic knots: torus translation, commutation, and
stabilization/destabilization. The stabilization moves divide further
into (essentially) four types, labeled \NW, \NE, \SW, and
\SE~ in the notation of \cite{bib:OST}. Of the Cromwell moves, torus
translation, commutation, and \NW~ and \SE~ de/stabilization
preserve Legendrian isotopy type, while \NE~ (resp.\ \SW)
stabilization is positive (resp.\ negative) stabilization in the
Legendrian category.

%*********************************************************************
\subsection{Braids and transverse knots}
\label{ssec:braids}

In some sense, the role played by grid diagrams for Legendrian knots
is played by braids for transverse knots. Let
\[
\xi_\rot = \ker(dz-y\,dx+x\,dy)
\]
denote the rotationally symmetric tight contact structure on $\R^3$.
There is a diffeomorphism $\phi$ of $\R^3$, given by $\phi(x,y,z) =
(x,2y,xy+z)$, that sends $\xi_\rot$ to $\xi_\std$. We can define
transverse knots for $\xi_\rot$ in the same way as transverse knots
for $\xi_\std$, and $\phi$ sends a knot tranverse to
$\xi_\rot$ to a knot transverse to $\xi_\std$. Thus we can view
any knot transverse to $\xi_\rot$ as a transverse knot in the sense of
Section~\ref{ssec:legknots}.

The closed curve $\{(\cos t,\sin t,0)\,|\,0\leq t\leq 2\pi\}$ traces
out an unknot $T_0$ transverse to $\xi_\rot$. We can then view
any braid $B$ as a transverse knot as follows. Embed the closure of
$B$ in a solid torus, and embed this solid torus as a small tubular
neighborhood of $T_0$. The braid then becomes a knot (or link) $T(B)$ in
$\R^3$ transverse to $\xi_\rot$, and can be mapped to a transverse
knot $\phi(T(B))$ in $(\R^3,\xi_\std)$ via the contactomorphism $\phi$.

Braids that are conjugate in the braid group yield transversely isotopic
knots. More interestingly, let a \textit{positive braid stabilization}
be the operation that replaces a braid $B\in B_n$ by $B\sigma_n \in
B_{n+1}$. Then we have the following result.

\begin{proposition}[Transverse Markov Theorem \cite{bib:OSh,bib:Wr}]
Let $B_1,B_2$ be braids. Then $T(B_1),T(B_2)$ are transversely
isotopic in $(\R^3,\xi_\rot)$ if and only if $B_1,B_2$ are related by
a sequence of braid conjugations and positive braid stabilizations and
destabilizations.
\end{proposition}

%*********************************************************************
\subsection{Translating between the three pictures}
\label{ssec:translate}

Given a braid word $B$, one can create the front of a Legendrian knot
$L(B)$ in a natural way as shown in Figure~\ref{fig:frontify}, cf.\
\cite{bib:Kal}. Draw the braid from left to right; each positive
crossing becomes part of the front in the obvious way, while each
negative crossing is represented by a zigzag and crossing in the
front. The corresponding left and right ends of the braid are then
connected through arcs with one left cusp and one right cusp apiece.

\begin{figure}
\centerline{
\includegraphics[height=1.5in]{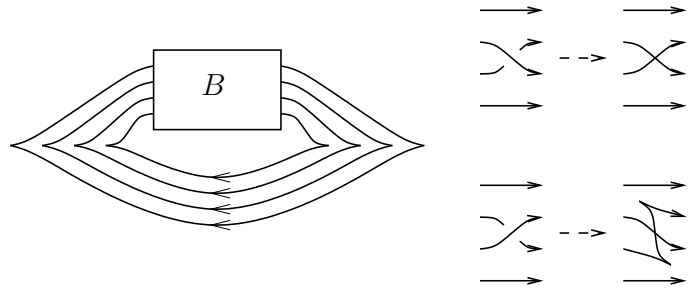}
}
\caption{
The Legendrian front $L(B)$ obtained from a braid word $B$.
}
\label{fig:frontify}
\end{figure}

We note that $L(B)$ is associated to a braid word and not a braid
isotopy class. If $B$ changes by braid isotopy,$L(B)$ changes
by a combination of Legendrian isotopies and negative Legendrian
de/stabilizations: replacing $\sigma_i\sigma_{i+1}\sigma_i$ by
$\sigma_{i+1}\sigma_i\sigma_{i+1}$ preserves Legendrian isotopy type,
while introducing $\sigma_i\sigma_i^{-1}$ or $\sigma_i^{-1}\sigma_i$
corresponds to one negative stabilization. Thus, for $B$ a braid,
$L(B)$ is only well-defined up to negative Legendrian
de/stabilization; however, the positive transverse pushoff
$L(B)_+$ constitutes a well-defined transverse isotopy class. 

In addition, it is straightforward to check that changing $B$ by
positive braid stabilization preserves the Legendrian isotopy, while
conjugating $B$ in the braid group changes $L(B)$ at most by negative
Legendrian de/stabilization (for the latter, see also the appendix on
the Legendrian satellite construction in \cite{bib:NT}). 
We conclude by the Transverse Markov Theorem that if two
braids $B_1,B_2$ close to transversely isotopic knots in
$(\R^3,\xi_\rot)$, then the positive transverse pushoffs $L(B_1)_+$,
$L(B_2)_+$ are transversely isotopic knots in $(\R^3,\xi_\std)$.
In fact, any braid $B$, viewed as a transverse knot
in $(\R^3,\xi_\rot)$, is the same as $L(B)_+$, viewed as a transverse
knot in $(\R^3,\xi_\std)$.

\begin{proposition}
Let $\phi$ be the contactomorphism between $(\R^3,\xi_\rot)$ and
$(\R^3,\xi_\std)$ from Section~\ref{ssec:braids}. If $B$ is a braid
and $T(B)$ is the transverse knot in $(\R^3,\xi_\rot)$ corresponding
to $B$, then $\phi(T(B))$ and $L(B)_+$ are transversely isotopic
knots in $(\R^3,\xi_\std)$.
\label{prop:comm1}
\end{proposition}

\begin{figure}
\centerline{
\includegraphics[height=2.4in]{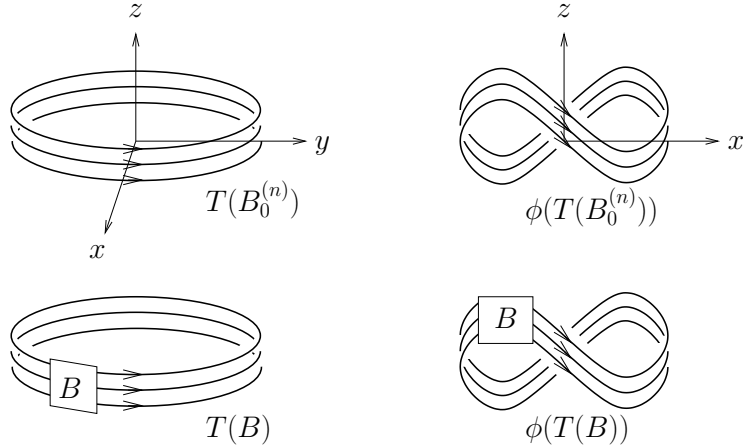}
}
\caption{
Transverse knots $T(B_0^{(n)}), T(B)$ in $\xi_\rot$ and the corresponding
fronts of transverse knots in $\xi_\std$.
}
\label{fig:transversify}
\end{figure}

\begin{proof}
Under $\phi$, the standard transverse unknot $T_0$ in $\xi_\rot$ maps
to an unknot whose front projection is a figure $8$. For any $n$, view
the trivial $n$-component braid $B_0^{(n)}$ as the transverse link in
$\xi_\rot$ defined by
\[
T(B_0^{(n)}) = \{(\cos t,\sin t,k\epsilon)\,|\,0\leq t\leq
2\pi,~k=0,\dots,n-1\}
\]
for small $\epsilon>0$. The
front projection of $\phi(T(B_0^{(n)}))$ is a collection of $n$ figure
$8$'s that differ by $\epsilon$ translations in the $z$ direction; see
Figure~\ref{fig:transversify}.

Let $B$ be a braid with $n$ strands. In cylindrical coordinates
$(r,\theta,z)$ on $\R^3$, we can choose $T(B)$ to agree with
$T(B_0^{(n)})$ except in a neighborhood of $\theta = 5\pi/4$, where the
entire braid $B$ lives. Then the front of $\phi(T(B))$ agrees with
$\phi(T(B_0^{(n)}))$ except in the braiding region near $(x,z) =
(-1/\sqrt{2},1/2)$.

\begin{figure}
\centerline{
\includegraphics[height=3in]{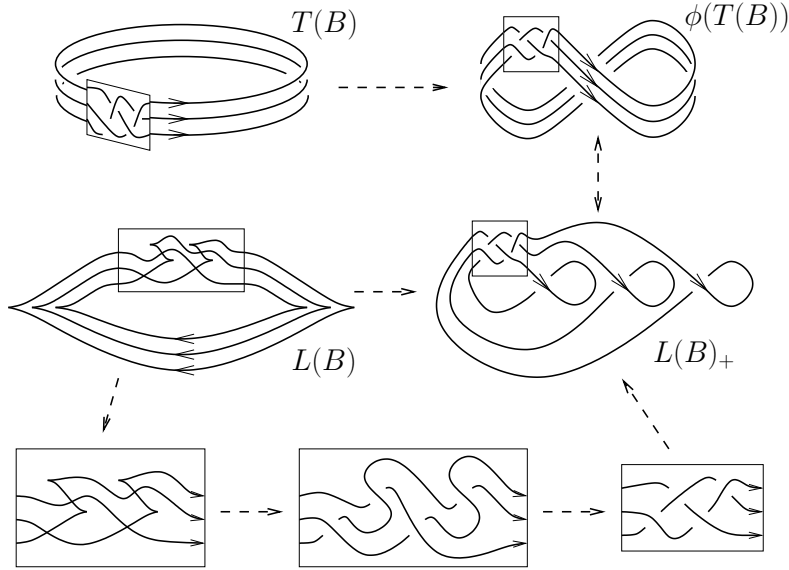}
}
\caption{
Transverse isotopy between $\phi(T(B))$ and $L(B)_+$. The top row is
as in Figure~\ref{fig:transversify}. The middle row shows the positive
transverse pushoff of $L(B)$, resulting in a transverse front that is
isotopic to the front for $\phi(T(B))$. The bottom row shows a detail
of the braiding region for the fronts of $L(B)$; $L(B)_+$; and
$L(B)_+$ after a transverse isotopy.
}
\label{fig:transversify2}
\end{figure}

In the braiding region for the front of $\phi(T(B))$, we can draw $B$
in the standard way, such that each strand goes from left to right
without vertical tangencies. We can then modify the front of
$\phi(T(B))$ by a transverse isotopy so that the $n$ figure $8$'s do
not intersect anywhere outside of the braiding region. The result is a
transverse front that is transversely isotopic to the front of
$L(B)_+$; see Figure~\ref{fig:transversify2}. The result follows.
\end{proof}

\begin{figure}
\centerline{
\includegraphics[height=2.5in]{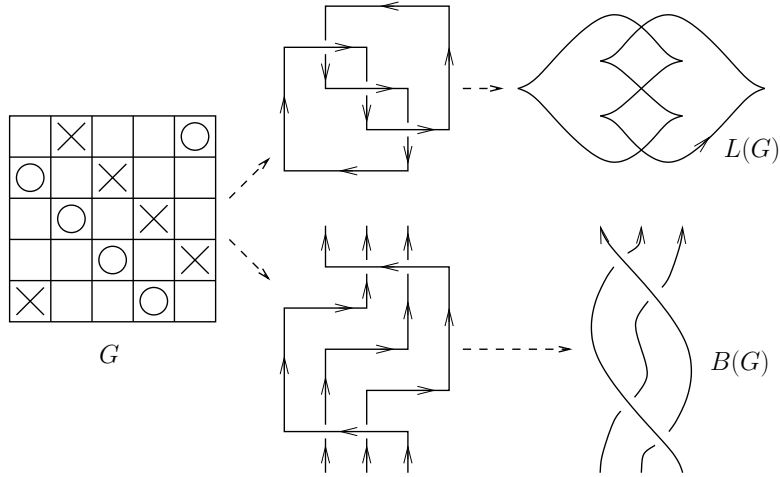}
}
\caption{
Constructing a Legendrian knot $L(G)$ and a braid $B(G)$ from a grid
diagram $G$. In this example, we read $B(G)$ bottom to top to get
$B(G) = \sigma_1^{-1}\sigma_2^{-1}\sigma_1^{-1}\sigma_2^{-1}$.
}
\label{fig:braidify}
\end{figure}

We next turn to the relation between grid diagrams and braids. Given
any grid diagram $G$, one can construct a braid $B(G)$ as
follows (cf.\ \cite{bib:Cro,bib:Dyn,bib:NOT}).
Connect $O$'s to $X$'s horizontally as usual. When an $O$
lies vertically above an $X$, connect them; when an $O$ lies
vertically under an $X$, draw two vertical line segments, one from the
$O$ down to the bottom of the grid diagram, one from the $X$ up to the
top of the grid diagram. Whenever two line segments cross, have the
horizontal segment cross over the vertical segment as before. We can
now orient all segments so that $O$'s point to $X$'s horizontally as
usual, and all vertical segments are oriented upwards. The result can
be viewed as a braid from the bottom of the grid diagram to the
top. See Figure~\ref{fig:braidify}. We remark that the closure of
$B(G)$ is isotopic to the knot given by $G$, and that any braid is $B(G)$
for some grid diagram $G$.

To a grid diagram $G$, we have now associated a Legendrian knot $L(G)$
and a braid $B(G)$. The following result is a compatibility result for
these two constructions as well as the construction $L(B)$.

\begin{proposition}
If $G$ is a grid diagram, then $L(G)_+$ and $L(B(G))_+$ are
transversely isotopic; that is, the Legendrian knots $L(G)$ and
$L(B(G))$ are related by Legendrian isotopy and negative de/stabilization.
\label{prop:comm2}
\end{proposition}

\begin{figure}
\centerline{
\includegraphics[height=3.5in]{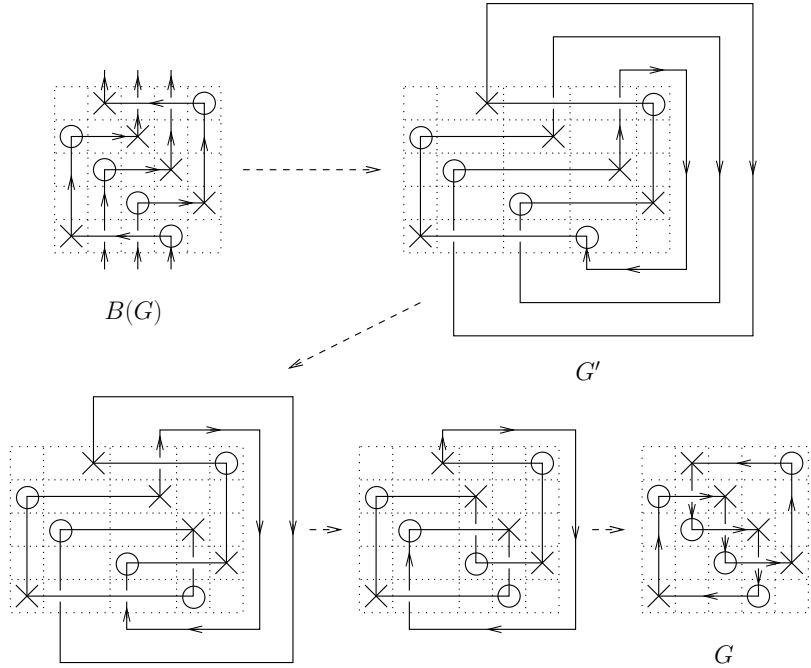}
}
\caption{
Destabilizing a grid diagram $G'$ obtained from $B(G)$ to recover
$G$. The diagram $G'$, viewed as a Legendrian knot, is itself some
number of negative destabilizations of $L(B(G))$ ($2$ in this case).
}
\label{fig:unbraidify}
\end{figure}

\begin{proof}
Let $G$ be an $n\times n$ grid diagram, and suppose $B(G)$ has $m$
strands. One can associate to $B(G)$ a natural $(n+2m)\times(n+2m)$
grid diagram $G'$ within which the original diagram $G$ appears, such
that $L(B(G))$ is Legendrian isotopic to $S_-^k(G')$ for some $k$; see
Figure~\ref{fig:unbraidify}. (More precisely, $k$ is the number of
appearances of subwords of the form $\sigma_i^{-1}\sigma_{i+1}^{-1}$
in $B(G)$.) There are $m$ ``braid'' parts of the new grid diagram,
each of which begins at an $X$, goes up 
out of $G$, curves around $G$ to the right, and ends at an $O$. We can
eliminate each of these parts in succession, via grid commutation and
two grid destabilizations, one a Legendrian isotopy and one a negative
Legendrian destabilization. The end result is $L(G)$. We conclude that
$L(B(G))$ is Legendrian isotopic to $S_-^{m+k}(L(G))$, and the proposition
follows.
\end{proof}

We remark that one can obtain another braid $B'(G)$ from a grid
diagram $G$ such that Proposition~\ref{prop:comm2} also holds. Instead of
forcing all vertical segments to point upwards, we instead force all
horizontal segments to point leftwards. This yields a braid $B'(G)$ by
reading from right to left. The two braids $B(G)$ and $B'(G)$ are
almost never identical or even conjugate, but they do represent the
same transverse knot.

To see this, define the \textit{diagonal mirror}
$G'$ of a grid diagram $G$ to be the grid diagram obtained by reflecting in
the main (upper left to lower right) diagonal and interchanging all
$X$'s and $O$'s. The diagonal mirror represents the same topological
knot as the original grid diagram, and in fact represents the same
Legendrian knot up to Legendrian isotopy. This follows from the fact
that the diagonal mirror replaces a Legendrian front in the $xz$ plane
with its reflection in the $z$ axis, corresponding to the
contactomorphism $(x,y,z) \mapsto (-x,-y,z)$ of $(\R^3,\xi_\std)$,
which is just a rotation in the $xy$ plane and preserves Legendrian
knots up to isotopy. Now $B'(G)$ is the same braid as
$B(G')$ and thus corresponds to the same transverse knot as $B(G)$.

We summarize the results from this section in the following diagram,
where $G$, $B$, $L$, and $T$ represent grid diagrams, braids,
Legendrian knots/links in $(\R^3,\xi_\std)$, and transverse
knots/links in $(\R^3,\xi_\std)$ respectively.
\[
\xymatrix@=3pc{
G \ar^{L(G)}[rr] \ar_{B(G)}[d] && L \ar^{L_+}[d] \\
B %\ar_{L(B)}[ur] 
\ar_{\phi(T(B)) = L(B)_+}[rr] && T
}
\]
The equality at the bottom of this diagram is
Proposition~\ref{prop:comm1}, while Proposition~\ref{prop:comm2}
states that the square commutes. 
%Proposition~\ref{prop:comm1} states that the bottom right triangle
%commutes. The top left triangle does not commute, but by
%Proposition~\ref{prop:comm2}, it commutes when composed with the
%positive transverse pushoff map $L \mapsto L_+$. It follows that the
%outside square commutes as well.

%*********************************************************************
%*********************************************************************
\section{Proof of Theorem~\ref{thm:main}}
\label{sec:proof}

For $a,b\geq 0$, we write the two relevant $4$-braids as
\begin{align*}
B_1(a,b) &= \sigma_3\sigma_2^{-2}\sigma_3^{2a+2}\sigma_2\sigma_3^{-1}
\sigma_1^{-1}\sigma_2\sigma_1^{2b+2} \\
B_2(a,b) &= \sigma_3\sigma_2^{-2}\sigma_3^{2a+2}\sigma_2\sigma_3^{-1}
\sigma_1^{2b+2}\sigma_2\sigma_1^{-1}.
\end{align*}
We now break the proof of Theorem~\ref{thm:main} into several parts.

\begin{proposition}
The knot given by the closure of $B_1(a,b)$ (or $B_2(a,b)$) is prime.
\end{proposition}

\begin{proof}
Let $K(a,b)$ denote the closure of $B_1(a,b)$.
Since braid index minus $1$ is additive under connected sum, if
$K(a,b)$ is composite, then $K(a,b) = T(2,2p+1) \# K'$ for some integer $p\neq 0,-1$,
where $T(2,2p+1)$ is the $(2,2p+1)$ torus knot. It thus suffices to show that $P(T(2,2p+1))(x,z)$ does not divide $P(K(a,b))(x,z)$, where $P(K)(x,z)$ is the HOMFLY-PT polynomial of $K$, defined by
$P(\begin{pspicture}(-.1,0)(0.5,0.45) \pscircle(0.2,0.1){0.2} \end{pspicture}) = 1$, $x P(\begin{pspicture}(-.1,0)(0.5,0.45)   \psline{->}(0,-.1)(.4,.3) \psline[border=2pt]{->}(0,.3)(.4,-.1) \end{pspicture}) - x^{-1} P(\begin{pspicture}(-.1,0)(0.5,0.45) \psline{->}(0,.3)(.4,-.1) \psline[border=2pt]{->}(0,-.1)(.4,.3) \end{pspicture}) = z P(\begin{pspicture}(-.1,0)(0.5,0.45) \psline[linearc=.2](0,.3)(.2,.2)(.4,.3) \psline[linearc=.2](0,-.1)(.2,0)(.4,-.1) \end{pspicture})$.

It is straightforward to check by induction that
\begin{align*}
P(K(a,b))(x,0) &= x^{-2a-2b-6}\left(-2a-2+(3a+3-b)x^2+(b-a)x^4\right) \\
P(T(2,2p+1))(x,0) &= x^{-2p-2}(-p+(p+1)x^2) \\
P(K(a,b))(x,2i) &= -(-x^2)^{-a-b-3}\left(2(1+a)(1+2b) \right. \\
& \qquad \left. +(3+3a+7b+8ab)x^2+(a+3b+4ab)x^4\right) \\
P(T(2,2p+1))(x,2i) &= -(-x^2)^{-p-1}(p+(p+1)x^2),
\end{align*}
where $i = \sqrt{-1}$. If $P(T(2,2p+1))$ divides $P(K(a,b))$, then the first two equations imply that $2a+2+(a+b+1)p=p^2$, while the last two equations imply that $2a+2+4b+4ab+(a+b+1)p=p^2$. When $a,b \geq 0$, it is easy to see that these identities can hold only when $a=b=0$ and $p=2$. In this case, $K(0,0)$ is (the mirror of) $10_{132}$ and hence prime.
\end{proof}

To apply the $\widehat{\theta}$ invariant to $B_1(a,b)$ and
$B_2(a,b)$, we need grid diagrams for both braids. It is possible to
create grid diagrams directly from the braids, but to facilitate the
computation we need particular diagrams.

\begin{proposition}
Let $G_1(a,b)$ and $G_2(a,b)$ be the $(2a+2b+9)\times(2a+2b+9)$ and
$(2a+2b+10)\times(2a+2b+10)$ grid diagrams depicted in
Figures~\ref{fig:gridbd3} and~\ref{fig:gridbd4}. Then $B(G_1(a,b))$
and $B(G_2(a,b))$ represent the same transverse knots as $B_1(a,b)$
and $B_2(a,b)$, respectively.
\end{proposition}

\begin{figure}[htbp]
	\centering
\begin{pspicture}(0,0)(6.5,7.2)
\psset{xunit=0.5cm,yunit=0.5cm, runit=0.5}
\psgrid[subgriddiv=1,gridlabels=0,linewidth=0.05mm](0,0)(13,13)
\psframe[linewidth=0.5mm](8,8)(13,13)
\psframe[linewidth=0.5mm, linestyle=dashed](2,4)(6,8)
\psline(0.75,4.25)(0.25,4.75)\psline(0.25,4.25)(0.75,4.75)
\psline(1.75,3.25)(1.25,3.75)\psline(1.25,3.25)(1.75,3.75)
\psline(2.75,5.25)(2.25,5.75)\psline(2.25,5.25)(2.75,5.75)
\psline(3.25,6.25)(3.75,6.75)\psline(3.75,6.25)(3.25,6.75)
\psline(4.25,7.25)(4.75,7.75)\psline(4.75,7.25)(4.25,7.75)
\psline(5.25,1.25)(5.75,1.75)\psline(5.75,1.25)(5.25,1.75)
\psline(6.25,0.25)(6.75,0.75)\psline(6.75,0.25)(6.25,0.75)
\psline(7.25,2.25)(7.75,2.75)\psline(7.75,2.25)(7.25,2.75)
\psline(8.25,8.25)(8.75,8.75)\psline(8.75,8.25)(8.25,8.75)
\psline(9.25,9.25)(9.75,9.75)\psline(9.75,9.25)(9.25,9.75)
\psline(10.25,10.25)(10.75,10.75)\psline(10.75,10.25)(10.25,10.75)
\psline(11.25,11.25)(11.75,11.75)\psline(11.75,11.25)(11.25,11.75)
\psline(12.25,12.25)(12.75,12.75)\psline(12.75,12.25)(12.25,12.75)
\pscircle(0.5,0.5){0.3}
\pscircle(1.5,11.5){0.3}
\pscircle(2.5,2.5){0.3}
\pscircle(3.5,4.5){0.3}
\pscircle(4.5,5.5){0.3}
\pscircle(5.5,6.5){0.3}
\pscircle(6.5,3.5){0.3}
\pscircle(7.5,12.5){0.3}
\pscircle(8.5,1.5){0.3}
\pscircle(9.5,7.5){0.3}
\pscircle(10.5,8.5){0.3}
\pscircle(11.5,9.5){0.3}
\pscircle(12.5,10.5){0.3}
\psdot*[dotsize=5pt](0,0)
\psdot*[dotsize=5pt](1,5)
\psdot*[dotsize=5pt](2,4)
\psdot*[dotsize=5pt](3,6)
\psdot*[dotsize=5pt](4,7)
\psdot*[dotsize=5pt](5,8)
\psdot*[dotsize=5pt](6,2)
\psdot*[dotsize=5pt](7,1)
\psdot*[dotsize=5pt](8,3)
\psdot*[dotsize=5pt](9,9)
\psdot*[dotsize=5pt](10,10)
\psdot*[dotsize=5pt](11,11)
\psdot*[dotsize=5pt](12,12)
\uput[u](10.5,13){$\overbrace{\ \ \ \ \ \ \ \ \ \ \ \ \ \ \ \ \ \ \ }^{2a+3}$}
\uput[u](4,13){$\overbrace{\ \ \ \ \ \ \ \ \ \ \ \ \ }^{2b+2}$}
\end{pspicture}
	\caption{The grid diagram $G_1(a,b)$ for $B_1(a,b)$, along
          with the state $x^+$. The solid
        box denotes a pattern consisting of $X$'s on the diagonal and
        $O$'s on the second subdiagonal; the dashed box has $X$'s on
        the first superdiagonal and $O$'s on the first subdiagonal.}
	\label{fig:gridbd3}
\end{figure}
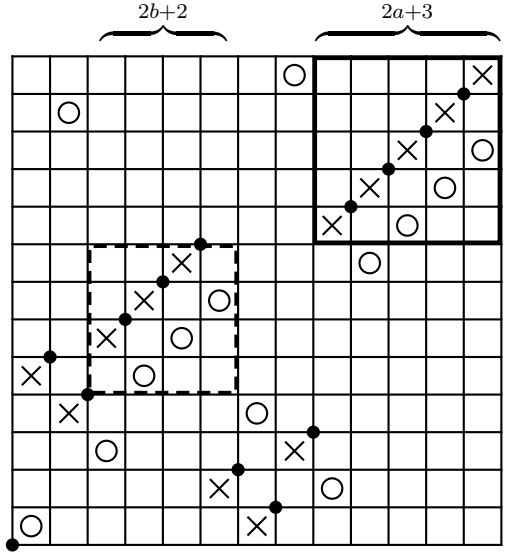

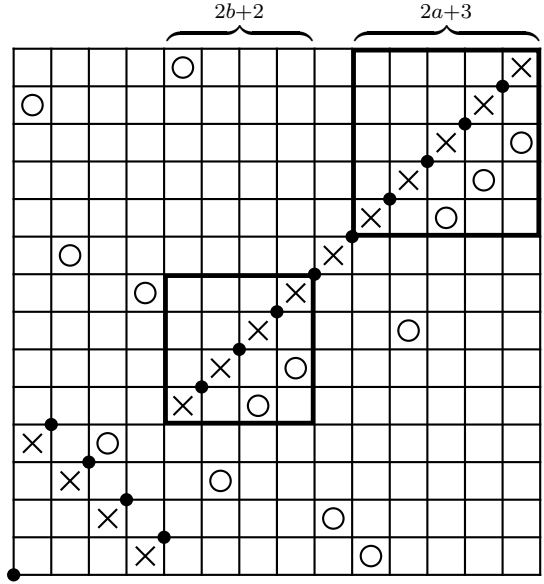
\begin{figure}
	\centering
\begin{pspicture}(0,0)(7,7.7)
\psset{xunit=0.5cm,yunit=0.5cm, runit=0.5}
\psgrid[subgriddiv=1,gridlabels=0,linewidth=0.05mm](0,0)(14,14)
\psframe[linewidth=0.5mm](9,9)(14,14)
\psframe[linewidth=0.5mm](4,4)(8,8)
\psline(0.75,3.25)(0.25,3.75)\psline(0.25,3.25)(0.75,3.75)
\psline(1.75,2.25)(1.25,2.75)\psline(1.25,2.25)(1.75,2.75)
\psline(2.75,1.25)(2.25,1.75)\psline(2.25,1.25)(2.75,1.75)
\psline(3.25,0.25)(3.75,0.75)\psline(3.75,0.25)(3.25,0.75)
\psline(4.25,4.25)(4.75,4.75)\psline(4.75,4.25)(4.25,4.75)
\psline(5.25,5.25)(5.75,5.75)\psline(5.75,5.25)(5.25,5.75)
\psline(6.25,6.25)(6.75,6.75)\psline(6.75,6.25)(6.25,6.75)
\psline(7.25,7.25)(7.75,7.75)\psline(7.75,7.25)(7.25,7.75)
\psline(8.25,8.25)(8.75,8.75)\psline(8.75,8.25)(8.25,8.75)
\psline(9.25,9.25)(9.75,9.75)\psline(9.75,9.25)(9.25,9.75)
\psline(10.25,10.25)(10.75,10.75)\psline(10.75,10.25)(10.25,10.75)
\psline(11.25,11.25)(11.75,11.75)\psline(11.75,11.25)(11.25,11.75)
\psline(12.25,12.25)(12.75,12.75)\psline(12.75,12.25)(12.25,12.75)
\psline(13.25,13.25)(13.75,13.75)\psline(13.75,13.25)(13.25,13.75)
\pscircle(0.5,12.5){0.3}
\pscircle(1.5,8.5){0.3}
\pscircle(2.5,3.5){0.3}
\pscircle(3.5,7.5){0.3}
\pscircle(4.5,13.5){0.3}
\pscircle(5.5,2.5){0.3}
\pscircle(6.5,4.5){0.3}
\pscircle(7.5,5.5){0.3}
\pscircle(8.5,1.5){0.3}
\pscircle(9.5,0.5){0.3}
\pscircle(10.5,6.5){0.3}
\pscircle(11.5,9.5){0.3}
\pscircle(12.5,10.5){0.3}
\pscircle(13.5,11.5){0.3}
\psdot*[dotsize=5pt](0,0)
\psdot*[dotsize=5pt](1,4)
\psdot*[dotsize=5pt](2,3)
\psdot*[dotsize=5pt](3,2)
\psdot*[dotsize=5pt](4,1)
\psdot*[dotsize=5pt](5,5)
\psdot*[dotsize=5pt](6,6)
\psdot*[dotsize=5pt](7,7)
\psdot*[dotsize=5pt](8,8)
\psdot*[dotsize=5pt](9,9)
\psdot*[dotsize=5pt](10,10)
\psdot*[dotsize=5pt](11,11)
\psdot*[dotsize=5pt](12,12)
\psdot*[dotsize=5pt](13,13)
%\psframe[fillstyle=hlines](1,3)(2,4)
%\psframe[fillstyle=hlines](2,2)(3,3)
%\psframe[fillstyle=hlines](3,1)(4,2)
\uput[u](11.5,13.8){$\overbrace{\ \ \ \ \ \ \ \ \ \ \ \ \ \ \ \ \ \ \ }^{2a+3}$}
\uput[u](6,13.8){$\overbrace{\ \ \ \ \ \ \ \ \ \ \ \ \ \ \ }^{2b+2}$}
\end{pspicture}
	\caption{The grid diagram $G_2(a,b)$ for $B_2(a,b)$.}
	\label{fig:gridbd4}
	\end{figure}

\begin{proof}
The braids $B(G_1(a,b))$ and $B(G_2(a,b))$ have a large number of
strands. It is easier to work with $B'(G_1(a,b))$ and $B'(G_2(a,b))$,
where $B'$ is the braid constructed by stipulating that all horizontal
segments point leftward; see Section~\ref{ssec:translate}. It was
shown in Section~\ref{ssec:translate} that $B(G)$ and $B'(G)$
represent the same transverse knot for any grid diagram $G$.

We readily calculate from the grid diagrams that $B'(G_1(a,b))$ and
$B'(G_2(a,b))$ are $4$-braids given by
\begin{align*}
B'(G_1(a,b)) &= \sigma_3^{2a+3} \sigma_2 \sigma_3^{-1} \sigma_1^{-2}
\sigma_2^{2b+1} \sigma_1 \sigma_2^{-1} \sigma_1 \\
B'(G_2(a,b)) &= \sigma_3^{2a+2} \sigma_2 \sigma_1 \sigma_3 \sigma_2 \sigma_1 \sigma_3 \sigma_2 \sigma_3^{2b+1}
\sigma_1^{-1}\sigma_2^{-1}\sigma_1^{-2}\sigma_2^{-1}\sigma_1^{-1}\sigma_2^{-1}.
\end{align*}
From relations in the braid group, we find that
\[
B'(G_2(a,b)) = \sigma_3^{2a+2} \sigma_2 \sigma_3^{-1} \sigma_1^{2b+2}
\sigma_2 \sigma_1^{-1} \sigma_3 \sigma_2^{-2}
\]
and thus $B'(G_2(a,b))$ is conjugate to $B_2(a,b)$.

The braids $B_1(a,b)$ and $B'(G_1(a,b))$ are not conjugate, but are
related by conjugation and exchange moves. For our purposes, we recognize
two exchange moves on $4$-braids, related by conjugation:
\begin{itemize}
\item
$\sigma_1$ exchange: $b_1\sigma_1b_2\sigma_1^{-1}b_3 \leftrightarrow
b_1\sigma_1^{-1}b_2\sigma_1b_3$, where $b_1,b_2,b_3$ are braids in the
subgroup generated by $\sigma_2,\sigma_3$;
\item
$\sigma_3$ exchange: $b_1\sigma_3b_2\sigma_3^{-1}b_3 \leftrightarrow
b_1\sigma_3^{-1}b_2\sigma_3b_3$, where $b_1,b_2,b_3$ are braids in the
subgroup generated by $\sigma_1,\sigma_2$.
\end{itemize}
Since an exchange move is a composition of conjugations, one positive
braid stabilization, and one positive destabilization, it does not
change the transverse type of the braid \cite{bib:BW}. Now we have
\begin{align*}
B'(G_1(a,b)) &\stackrel{\mathrm{conj}\,\sigma_1}{\longrightarrow}
\sigma_2^{-1} \sigma_3^{-1} \sigma_2^{2a+3} \sigma_1 \sigma_2^{-1}
\sigma_3 \sigma_2 \sigma_1^{-1} \sigma_2^{2b+1} \sigma_1 \sigma_2^{-1}
\\
&\stackrel{\mathrm{exch}\,\sigma_3}{\longrightarrow}
\sigma_2^{-1} \sigma_3 \sigma_2^{2a+3} \sigma_1
\sigma_2^{-1} \sigma_3^{-1} \sigma_2 \sigma_1^{-1} \sigma_2^{2b+1}
\sigma_1 \sigma_2^{-1} \\
& \qquad = \sigma_2^{-1} \sigma_3 \sigma_2^{2a+3} \sigma_3
\sigma_2^{-1} \sigma_1^{-1} \sigma_2 \sigma_3^{-1} \sigma_2^{2b+1}
\sigma_1 \sigma_2^{-1} \\
&\stackrel{\mathrm{exch}\,\sigma_1}{\longrightarrow} \sigma_2^{-1}
\sigma_3 \sigma_2^{2a+3} \sigma_3
\sigma_2^{-1} \sigma_1 \sigma_2 \sigma_3^{-1} \sigma_2^{2b+1}
\sigma_1^{-1} \sigma_2^{-1} \\
&\stackrel{\mathrm{conj}\,\sigma_3}{\longrightarrow} B_1(a,b),
\end{align*}
where ``conj $\sigma_k$'' is conjugation by $\sigma_k$, $B \mapsto
\sigma_k B \sigma_k^{-1}$.
\end{proof}

We now use $\widehat{\theta}$ to show that $G_1(a,b)$ and $G_2(a,b)$
are of different transverse types; more precisely, if we define
transverse knots
\[
T_i(a,b) = L(G_i(a,b))_+ =
\phi(T(B_i(a,b)))
\]
in $(\R^3,\xi_\std)$ for $i=1,2$, then $T_1(a,b)$ and $T_2(a,b)$
are not transversely isotopic. This will complete the proof of
Theorem~\ref{thm:main}.

%We use the notation of \cite{bib:NOT}.
Recall from \cite{bib:OST} that if $T$ is a
transverse knot, then $\widehat{\theta}(T)$ is an element of the knot
Floer homology $\HFKhat(m(T))$, where $m(T)$ is the topological mirror
of $T$. If $T = L(G)_+$ for a grid diagram $G$, then $\HFKhat(m(T))$
can be combinatorially computed from $G$ as in \cite{bib:MOS}. It is
easier to consider a variant $\HFKtilde(m(T))$ of $\HFKhat(m(T))$, in
which there is a corresponding element $\widetilde{\theta}(T)$
($=j_*(\widehat{\theta})$ in \cite{bib:NOT}); then $\widetilde{\theta}(T)
= 0$ if and only if $\widehat{\theta}(T) = 0$ \cite[section
4]{bib:NOT}.

We assume some familiarity with the combinatorial definition of
$\HFKtilde$ over $\Z/2$ from \cite{bib:MOS} (or \cite{bib:NOT,bib:OST}).
If $T = L(G)_+$ for an $n\times n$ grid diagram $G$, then
the complex $\CFKtilde(m(T))$ is generated by $n!$ states labeled by
permutations of $\{1,\dots,n\}$. A state $(\pi(1),\dots,\pi(n))$ can
be depicted in the grid as a set of $n$ points $\{(i,\pi(i))\}$, where
$(i,j)$ is the intersection of vertical line $i$ and horizontal line
$j$; here the vertical (resp.\ horizontal) lines are numbered left to
right (resp.\ bottom to top) starting with $1$. The differential
$\partial$ on $\CFKtilde(m(T))$ is represented pictorially by
\[
\partial \left( \ \begin{pspicture}(0,-.1)(0.5,0.25)
    \psset{xunit=0.5cm,yunit=0.5cm, runit=0.5}
\psframe(0,-.5)(1,.5) \psdot*[dotsize=3pt](0,-.5)
\psdot*[dotsize=3pt](1,.5) \end{pspicture} \ \right) =
\sum \quad \begin{pspicture}(0,-.1)(0.5,0.25)
  \psset{xunit=0.5cm,yunit=0.5cm, runit=0.5}
\psframe(0,-.5)(1,.5) \psdot*[dotsize=3pt](1,-.5)
\psdot*[dotsize=3pt](0,.5) \end{pspicture} \quad ,
\]
where the sum is over all rectangles not containing any $X$'s, $O$'s,
or other points in the state. If $\partial(y)$ contains $x$ as a term,
then we write $x \leftarrow y$ and $y \rightarrow x$. The transverse
invariant $\widetilde{\theta}$ is the image in $\HFKtilde(m(T))$ of
the state $x^+$ given by the upper right corners of the $X$'s.

\begin{proposition}
%Let $T_1(a,b) = L(G_1(a,b))_+$ for any $a,b \geq 0$. Then
We have $\widetilde{\theta}(T_1(a,b)) = 0$ and hence
$\widehat{\theta}(T_1(a,b)) = 0$.
\label{prop:t1}
\end{proposition}

\begin{figure}
	\centering
	\subfigure[$y_1$]{\begin{pspicture}(0,0)(3.25,3.4)
\psset{xunit=0.25cm,yunit=0.25cm, runit=0.25}
\psframe[linewidth=0.5mm](8,8)(13,13) \psframe[linewidth=0.5mm, linestyle=dashed](2,4)(6,8)
\psline(0.75,4.25)(0.25,4.75)\psline(0.25,4.25)(0.75,4.75)\psline(1.75,3.25)(1.25,3.75)\psline(1.25,3.25)(1.75,3.75)
\psline(2.75,5.25)(2.25,5.75)\psline(2.25,5.25)(2.75,5.75)\psline(3.25,6.25)(3.75,6.75)\psline(3.75,6.25)(3.25,6.75)
\psline(4.25,7.25)(4.75,7.75)\psline(4.75,7.25)(4.25,7.75)\psline(5.25,1.25)(5.75,1.75)\psline(5.75,1.25)(5.25,1.75)
\psline(6.25,0.25)(6.75,0.75)\psline(6.75,0.25)(6.25,0.75)\psline(7.25,2.25)(7.75,2.75)\psline(7.75,2.25)(7.25,2.75)
\psline(8.25,8.25)(8.75,8.75)\psline(8.75,8.25)(8.25,8.75)\psline(9.25,9.25)(9.75,9.75)\psline(9.75,9.25)(9.25,9.75)
\psline(10.25,10.25)(10.75,10.75)\psline(10.75,10.25)(10.25,10.75)\psline(11.25,11.25)(11.75,11.75)\psline(11.75,11.25)(11.25,11.75)\psline(12.25,12.25)(12.75,12.75)\psline(12.75,12.25)(12.25,12.75)
\pscircle(0.5,0.5){0.3}\pscircle(1.5,11.5){0.3}\pscircle(2.5,2.5){0.3}\pscircle(3.5,4.5){0.3}\pscircle(4.5,5.5){0.3}
\pscircle(5.5,6.5){0.3}\pscircle(6.5,3.5){0.3}\pscircle(7.5,12.5){0.3}\pscircle(8.5,1.5){0.3}\pscircle(9.5,7.5){0.3}
\pscircle(10.5,8.5){0.3}\pscircle(11.5,9.5){0.3}\pscircle(12.5,10.5){0.3}
\psframe[fillstyle=solid, fillcolor=lightgray](6,1)(7,2)
\psframe[fillstyle=solid, fillcolor=lightgray](5,8)(6,13) \psframe[fillstyle=solid, fillcolor=lightgray](5,0)(6,1)
\psdots*[dotsize=.6](0,0)(1,5)(2,4)(5,8)(6,1)(7,2)(8,3)
\psdots*[dotsize=.6](3,6)(4,7)(9,9)(10,10)(11,11)(12,12) \psgrid[subgriddiv=1,gridlabels=0,linewidth=0.05mm](0,0)(13,13)
\end{pspicture} }
  \subfigure[$y_2$]{\begin{pspicture}(0,0)(3.25,3.4) \psset{xunit=0.25cm,yunit=0.25cm, runit=0.25}
\psframe[linewidth=0.5mm](8,8)(13,13) \psframe[linewidth=0.5mm, linestyle=dashed](2,4)(6,8)
\psline(0.75,4.25)(0.25,4.75)\psline(0.25,4.25)(0.75,4.75)\psline(1.75,3.25)(1.25,3.75)\psline(1.25,3.25)(1.75,3.75)
\psline(2.75,5.25)(2.25,5.75)\psline(2.25,5.25)(2.75,5.75)\psline(3.25,6.25)(3.75,6.75)\psline(3.75,6.25)(3.25,6.75)
\psline(4.25,7.25)(4.75,7.75)\psline(4.75,7.25)(4.25,7.75)\psline(5.25,1.25)(5.75,1.75)\psline(5.75,1.25)(5.25,1.75)
\psline(6.25,0.25)(6.75,0.75)\psline(6.75,0.25)(6.25,0.75)\psline(7.25,2.25)(7.75,2.75)\psline(7.75,2.25)(7.25,2.75)
\psline(8.25,8.25)(8.75,8.75)\psline(8.75,8.25)(8.25,8.75)\psline(9.25,9.25)(9.75,9.75)\psline(9.75,9.25)(9.25,9.75)
\psline(10.25,10.25)(10.75,10.75)\psline(10.75,10.25)(10.25,10.75)\psline(11.25,11.25)(11.75,11.75)\psline(11.75,11.25)(11.25,11.75)\psline(12.25,12.25)(12.75,12.75)\psline(12.75,12.25)(12.25,12.75)
\pscircle(0.5,0.5){0.3}\pscircle(1.5,11.5){0.3}\pscircle(2.5,2.5){0.3}\pscircle(3.5,4.5){0.3}\pscircle(4.5,5.5){0.3}
\pscircle(5.5,6.5){0.3}\pscircle(6.5,3.5){0.3}\pscircle(7.5,12.5){0.3}\pscircle(8.5,1.5){0.3}\pscircle(9.5,7.5){0.3}
\pscircle(10.5,8.5){0.3}\pscircle(11.5,9.5){0.3}\pscircle(12.5,10.5){0.3}
\psframe[fillstyle=solid, fillcolor=lightgray](1,4)(2,5)
\psframe[fillstyle=solid, fillcolor=lightgray](8,3)(13,4) \psframe[fillstyle=solid, fillcolor=lightgray](0,3)(1,4) \psdots*[dotsize=.6](0,0)(1,4)(2,5)(5,1)(6,8)(7,2)(8,3)
\psdots*[dotsize=.6](3,6)(4,7)(9,9)(10,10)(11,11)(12,12) \psgrid[subgriddiv=1,gridlabels=0,linewidth=0.05mm](0,0)(13,13)
\end{pspicture} }
	\subfigure[$y_3$]{\begin{pspicture}(0,0)(3.25,3.4)
\psset{xunit=0.25cm,yunit=0.25cm, runit=0.25}
\psframe[linewidth=0.5mm](8,8)(13,13) \psframe[linewidth=0.5mm, linestyle=dashed](2,4)(6,8)
\psline(0.75,4.25)(0.25,4.75)\psline(0.25,4.25)(0.75,4.75)\psline(1.75,3.25)(1.25,3.75)\psline(1.25,3.25)(1.75,3.75)
\psline(2.75,5.25)(2.25,5.75)\psline(2.25,5.25)(2.75,5.75)\psline(3.25,6.25)(3.75,6.75)\psline(3.75,6.25)(3.25,6.75)
\psline(4.25,7.25)(4.75,7.75)\psline(4.75,7.25)(4.25,7.75)\psline(5.25,1.25)(5.75,1.75)\psline(5.75,1.25)(5.25,1.75)
\psline(6.25,0.25)(6.75,0.75)\psline(6.75,0.25)(6.25,0.75)\psline(7.25,2.25)(7.75,2.75)\psline(7.75,2.25)(7.25,2.75)
\psline(8.25,8.25)(8.75,8.75)\psline(8.75,8.25)(8.25,8.75)\psline(9.25,9.25)(9.75,9.75)\psline(9.75,9.25)(9.25,9.75)
\psline(10.25,10.25)(10.75,10.75)\psline(10.75,10.25)(10.25,10.75)\psline(11.25,11.25)(11.75,11.75)\psline(11.75,11.25)(11.25,11.75)\psline(12.25,12.25)(12.75,12.75)\psline(12.75,12.25)(12.25,12.75)
\pscircle(0.5,0.5){0.3}\pscircle(1.5,11.5){0.3}\pscircle(2.5,2.5){0.3}\pscircle(3.5,4.5){0.3}\pscircle(4.5,5.5){0.3}
\pscircle(5.5,6.5){0.3}\pscircle(6.5,3.5){0.3}\pscircle(7.5,12.5){0.3}\pscircle(8.5,1.5){0.3}\pscircle(9.5,7.5){0.3}
\pscircle(10.5,8.5){0.3}\pscircle(11.5,9.5){0.3}\pscircle(12.5,10.5){0.3}
\psframe[fillstyle=solid, fillcolor=lightgray](6,4)(8,8) \psdots[dotsize=.6](0,0)(1,3)(2,5)(5,1)(6,4)(7,2)(8,8)
\psdots[dotsize=.6](3,6)(4,7)(9,9)(10,10)(11,11)(12,12)  \psgrid[subgriddiv=1,gridlabels=0,linewidth=0.05mm](0,0)(13,13)
\end{pspicture} }
	\caption{Three grid states of $G_1(a,b)$.
The shaded boxes indicate terms contributing to the differentials of
these states.}
	\label{fig:threestate}
\end{figure}
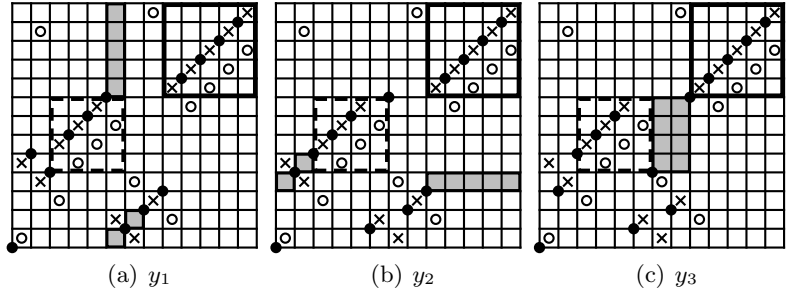

\begin{proof}
Write $e=2b+7$ and $f=2a+2b+9$; then the state $x^+$ is the
permutation $(1,6,5,7,8,\ldots,e,3,2,4,e+1,\ldots,f) =
(1,6,5,(\ldots)_1,e,3,2,4,(\ldots)_2)$, where $(\ldots)_1 =
7,\ldots,e-1$ and $(\ldots)_2 = e+1,\ldots,f$.
Define three further states
\begin{align*}
y_1 &=  (1,6,5,(\ldots)_1,e,2,3,4,(\ldots)_2)	\\
y_2 &=  (1,5,6,(\ldots)_1,2,e,3,4,(\ldots)_2)  \\
y_3 &=  (1,4,6,(\ldots)_1,2,5,3,e,(\ldots)_2).
\end{align*}
It is then easy to check (Figure~\ref{fig:threestate}) that
\begin{align*}
\partial(y_1) &= \ (1,6,5,(\ldots)_1,e,3,2,4,(\ldots)_2) +
(1,6,5,(\ldots)_1,2,e,3,4,(\ldots)_2) \\
\partial(y_2) &= \ (1,6,5,(\ldots)_1,2,e,3,4,(\ldots)_2) +
(1,4,6,(\ldots)_1,2,e,3,5,(\ldots)_2) \\
\partial(y_3) &= \ (1,4,6,(\ldots)_1,2,e,3,5,(\ldots)_2),
\end{align*}
and so
\[
\partial(y_1 + y_2 + y_3) = 	(1,6,5,(\ldots)_1,e,3,2,4,(\ldots)_2) = x^+.
\]
 Thus $x^+$ is null-homologous and $\widetilde{\theta}(T_1(a,b)) = 0$.
\end{proof}

\begin{proposition}
%Let $T_2(a,b) = L(G_2(a,b))_+$ for any $a,b \geq 0$. Then
We have $\widetilde{\theta}(T_2(a,b)) \neq 0$ and hence
$\widehat{\theta}(T_2(a,b)) \neq 0$.
\label{prop:t2}
\end{proposition}

\begin{proof}
We will only need to consider states of the form
\[
(\pi(1),\pi(2),\pi(3),\pi(4),\pi(5),6,7,8,\dots,2a+2b+11),
\]
which we abbreviate as $\pi(1),\pi(2),\pi(3),\pi(4),\pi(5)$. In this
notation, $x^+$ for $G_2(a,b)$ is $1,5,4,3,2$. To determine whether
$x^+$ is null-homologous, we work as in \cite{bib:NOT}.

Let $\A_0 =
\emptyset$ and $\B_0 = \{x^+\}$, and inductively define $\A_k,\B_k$ as
follows: let
\begin{align*}
\A_k &= \{\text{states } x\not\in \A_{k-1}
\, | \, y\rightarrow x \text{ for some } x\in\B_{k-1} \} \\
\B_k &= \{\text{states } y\not\in \B_{k-1}
\, | \, y\rightarrow x \text{ for some } x\in\A_k\}.
\end{align*}
This process terminates at some point; let $A,B$ be the free vector
spaces over $\Z/2$ generated by elements of all $\A_k,\B_k$
respectively, and let $C' = A \oplus B$. Then the differential
$\partial$ on $\CFKtilde$ induces a map $D\colon\thinspace A\to B$,
and $x^+$ is null-homologous in $\CFKtilde$ if and only if $x^+\in B$
is in the image of $D$.

For the grid $G_2(a,b)$, we successively calculate the $\A_k$ and
$\B_k$ as follows; see the appendix for the relevant pictures.
\begin{center}
$ x^+ =\  1,5,4,3,2 \ \  \mathfrak{B}_0 = \{x^+\} $  \\

$ \left. \begin{array}{rl}
y_1 =\  1,4,5,3,2 &  y_1 \rightarrow x^+ \\
y_2 =\ 1,5,3,4,2 & y_2 \rightarrow x^+ \\
y_3 =\ 1,5,4,2,3 & y_3 \rightarrow x^+ \end{array} \right\}
\mathfrak{A}_1 = \{y_1,y_2,y_3\} $ \linebreak[4]

$ \left. \begin{array}{rl}
x_1 =\ 4,1,5,3,2 & x_1 \leftarrow y_1 \\
x_2 =\ 3,5,1,4,2 & x_2 \leftarrow y_2 \\
x_3 =\ 2,5,4,1,3 & x_3 \leftarrow y_3  \end{array} \right\}
\mathfrak{B}_1 = \{x_1,x_2,x_3\}$ \linebreak[4]

$ \left. \begin{array}{rl}
y_4 =\ 4,1,5,2,3 & y_4 \rightarrow x_1 \\
y_5 =\ 3,5,1,2,4 & y_5 \rightarrow x_2 \\
y_6 =\ 2,4,5,1,3 & y_6 \rightarrow x_3 \\
y_7 =\ 2,5,3,1,4 & y_7 \rightarrow x_3 \end{array}
\right\}
\mathfrak{A}_2 = \{y_4,y_5,y_6,y_7\}$ \linebreak[4]

$ \left. \begin{array}{rl}
x_4 =\ 4,2,5,1,3 & x_4 \leftarrow y_4,y_6 \\
x_5 =\ 3,5,2,1,4 & x_5 \leftarrow y_5,y_7 \\
{x_6} =\ 4,1,2,5,3 & x_6 \leftarrow y_4 \\
{x_7} =\ 2,4,1,5,3 & x_7 \leftarrow y_6 \end{array} \right\}
\mathfrak{B}_2 = \{x_4,x_5,x_6,x_7\} $ \linebreak[4]
$ \left. \begin{array}{rl}
y_8 =\ 1,4,2,5,3 & y_8 \rightarrow x_6,x_7 \\
y_9 =\ 4,1,2,3,5 & y_9 \rightarrow x_6 \\
y_{10} =\ 5,1,2,4,3 & y_{10} \rightarrow x_6 \\
y_{11} =\ 2,4,1,3,5 & y_{11} \rightarrow x_7 \end{array} \right\}
\mathfrak{A}_3 = \{y_8,y_9,y_{10},y_{11}\} $ \linebreak[4]

$ \left. \begin{array}{rl}
x_8 =\ 4,2,1,3,5 & x_8 \leftarrow y_9,y_{11} \\
x_9 =\ 5,2,1,4,3 & x_9 \leftarrow y_{10} \end{array} \right\}
\mathfrak{B}_3 = \{x_8,x_9\} $ \linebreak[4]

$ \left. \begin{array}{rl}
y_{12} =\ 5,2,1,3,4 & y_{12} \rightarrow x_8,x_9  \end{array} \right\}
\mathfrak{A}_4 = \{y_{12}\}. $
\end{center}
We have $\mathfrak{B}_4 = \mathfrak{A}_5 = \emptyset$ and the process
terminates here.

The subcomplex $C$ is generated by
$x^+,x_1,\dots,x_9,y_1,\dots,y_{12}$. Below is the matrix for the
adjoint (transpose) of the map $D\colon\thinspace A\to B$, with rows
corresponding to
$x^+,x_1,\dots,x_9$ and columns corresponding to $y_1,\dots,y_{12}$, along
with its row reduction:
\[ \left(
\begin{smallmatrix}
1&1&0&0&0&0&0&0&0&0\\
1&0&1&0&0&0&0&0&0&0\\
1&0&0&1&0&0&0&0&0&0\\
0&1&0&0&1&0&1&0&0&0\\
0&0&1&0&0&1&0&0&0&0\\
0&0&0&1&1&0&0&1&0&0\\
0&0&0&1&0&1&0&0&0&0\\
0&0&0&0&0&0&1&1&0&0 \\
0&0&0&0&0&0&1&0&1&0 \\
0&0&0&0&0&0&1&0&0&1 \\
0&0&0&0&0&0&0&1&1&0 \\
0&0&0&0&0&0&0&0&1&1 \end{smallmatrix} \right) \stackrel{\text{row
reduce}}{\longrightarrow}
\left( \begin{smallmatrix}
1&0&0&0&0&1&0&0&0&0\\
0&1&0&0&0&1&0&0&0&0\\
0&0&1&0&0&1&0&0&0&0\\
0&0&0&1&0&1&0&0&0&0\\
0&0&0&0&1&1&0&0&0&1\\
0&0&0&0&0&0&1&0&0&1\\
0&0&0&0&0&0&0&1&0&1\\
0&0&0&0&0&0&0&0&1&1\\
0&0&0&0&0&0&0&0&0&0\\
0&0&0&0&0&0&0&0&0&0\\
0&0&0&0&0&0&0&0&0&0\\
0&0&0&0&0&0&0&0&0&0\\
 \end{smallmatrix} \right). \]
We see that $[1,0,0,0,0,0,0,0,0,0]$ is not in the row
space of the row-reduced matrix, and hence $x^+$ is not in the image
of $D$. We conclude that
$\widetilde{\theta}(T_2(a,b))$ is not null-homologous.
\end{proof}

Propositions~\ref{prop:t1} and~\ref{prop:t2} show that $T_1(a,b)$ and
$T_2(a,b)$ are different as transverse knots, and
Theorem~\ref{thm:main} follows.

\vspace{0.5in}

%*********************************************************************
%*********************************************************************

\vspace{0.5in}

%*********************************************************************
%*********************************************************************

\section*{Appendix: Grid states for $G_2(a,b)$}

On the next two pages, we depict the grid states $x_i,y_i$ from the
proof of Proposition~\ref{prop:t2}. For each state, the rectangles
comprising the differential for the $y$ states, or the adjoint differential for
the $x$ states, are shaded, with darker shading for rectangles from
previous states and lighter shading for the others. Some rectangles
overlap.
%Rectangles representing relevant
%contributions to the differential are shaded. The darker shaded
%rectangles correspond to empty rectangles from previously defined
%states. On the other hand, the lighter shaded rectangles are new
%empty rectangles.

\vspace{0.5in}

\begin{figure}[htbp] \centering

\subfigure[$x^+$]{ \begin{pspicture}(0,0)(4.2,4.2) \psset{xunit=0.3cm,yunit=0.3cm, runit=0.3}
\psframe[fillstyle=solid, fillcolor=lightgray](1,4)(2,3) \psframe[fillstyle=solid, fillcolor=lightgray](2,3)(3,2) \psframe[fillstyle=solid, fillcolor=lightgray](3,2)(4,1) \psgrid[subgriddiv=1,gridlabels=0,linewidth=0.05mm](0,0)(14,14)
\psframe[linewidth=0.5mm](9,9)(14,14) \psframe[linewidth=0.5mm](4,4)(8,8)
\psline(0.75,3.25)(0.25,3.75)\psline(0.25,3.25)(0.75,3.75)\psline(1.75,2.25)(1.25,2.75)\psline(1.25,2.25)(1.75,2.75)
\psline(2.75,1.25)(2.25,1.75)\psline(2.25,1.25)(2.75,1.75)\psline(3.25,0.25)(3.75,0.75)\psline(3.75,0.25)(3.25,0.75)
\psline(4.25,4.25)(4.75,4.75)\psline(4.75,4.25)(4.25,4.75)\psline(5.25,5.25)(5.75,5.75)\psline(5.75,5.25)(5.25,5.75)
\psline(6.25,6.25)(6.75,6.75)\psline(6.75,6.25)(6.25,6.75)\psline(7.25,7.25)(7.75,7.75)\psline(7.75,7.25)(7.25,7.75)
\psline(8.25,8.25)(8.75,8.75)\psline(8.75,8.25)(8.25,8.75)\psline(9.25,9.25)(9.75,9.75)\psline(9.75,9.25)(9.25,9.75)
\psline(10.25,10.25)(10.75,10.75)\psline(10.75,10.25)(10.25,10.75)
\psline(11.25,11.25)(11.75,11.75)\psline(11.75,11.25)(11.25,11.75)
\psline(12.25,12.25)(12.75,12.75)\psline(12.75,12.25)(12.25,12.75)
\psline(13.25,13.25)(13.75,13.75)\psline(13.75,13.25)(13.25,13.75)
\pscircle(0.5,12.5){0.3}\pscircle(1.5,8.5){0.3}\pscircle(2.5,3.5){0.3}\pscircle(3.5,7.5){0.3}\pscircle(4.5,13.5){0.3}\pscircle(5.5,2.5){0.3}\pscircle(6.5,4.5){0.3}\pscircle(7.5,5.5){0.3}\pscircle(8.5,1.5){0.3}\pscircle(9.5,0.5){0.3}\pscircle(10.5,6.5){0.3}\pscircle(11.5,9.5){0.3}\pscircle(12.5,10.5){0.3}\pscircle(13.5,11.5){0.3}
\psdots*[dotsize=.5](0,0)(1,4)(2,3)(3,2)(4,1)
\psdots*[dotsize=.5](5,5)(6,6)(7,7)(8,8)(9,9)(10,10)(11,11)(12,12)(13,13)
\end{pspicture} }\subfigure[$y_1$]{ \begin{pspicture}(0,0)(4.2,4.2) \psset{xunit=0.3cm,yunit=0.3cm, runit=0.3}
\psframe[fillstyle=solid, fillcolor=lightgray](0,0)(1,3)
\psframe[fillstyle=solid, fillcolor=darkgray](1,3)(2,4) \psgrid[subgriddiv=1,gridlabels=0,linewidth=0.05mm](0,0)(14,14)
\psframe[linewidth=0.5mm](9,9)(14,14)\psframe[linewidth=0.5mm](4,4)(8,8)
\psline(0.75,3.25)(0.25,3.75)\psline(0.25,3.25)(0.75,3.75)\psline(1.75,2.25)(1.25,2.75)\psline(1.25,2.25)(1.75,2.75)
\psline(2.75,1.25)(2.25,1.75)\psline(2.25,1.25)(2.75,1.75)\psline(3.25,0.25)(3.75,0.75)\psline(3.75,0.25)(3.25,0.75)
\psline(4.25,4.25)(4.75,4.75)\psline(4.75,4.25)(4.25,4.75)\psline(5.25,5.25)(5.75,5.75)\psline(5.75,5.25)(5.25,5.75)
\psline(6.25,6.25)(6.75,6.75)\psline(6.75,6.25)(6.25,6.75)\psline(7.25,7.25)(7.75,7.75)\psline(7.75,7.25)(7.25,7.75)
\psline(8.25,8.25)(8.75,8.75)\psline(8.75,8.25)(8.25,8.75)\psline(9.25,9.25)(9.75,9.75)\psline(9.75,9.25)(9.25,9.75)
\psline(10.25,10.25)(10.75,10.75)\psline(10.75,10.25)(10.25,10.75)\psline(11.25,11.25)(11.75,11.75)\psline(11.75,11.25)(11.25,11.75)\psline(12.25,12.25)(12.75,12.75)\psline(12.75,12.25)(12.25,12.75)
\psline(13.25,13.25)(13.75,13.75)\psline(13.75,13.25)(13.25,13.75)
\pscircle(0.5,12.5){0.3}
\pscircle(1.5,8.5){0.3}
\pscircle(2.5,3.5){0.3}
\pscircle(3.5,7.5){0.3}
\pscircle(4.5,13.5){0.3}
\pscircle(5.5,2.5){0.3}
\pscircle(6.5,4.5){0.3}
\pscircle(7.5,5.5){0.3}
\pscircle(8.5,1.5){0.3}
\pscircle(9.5,0.5){0.3}
\pscircle(10.5,6.5){0.3}
\pscircle(11.5,9.5){0.3}
\pscircle(12.5,10.5){0.3}
\pscircle(13.5,11.5){0.3}
\psdots*[dotsize=.5](0,0)(1,3)(2,4)(3,2)(4,1)
\psdots*[dotsize=.5](5,5)(6,6)(7,7)(8,8)(9,9)(10,10)(11,11)(12,12)(13,13)
\end{pspicture}  }\subfigure[$y_2$]{ \begin{pspicture}(0,0)(4.2,4.2) \psset{xunit=0.3cm,yunit=0.3cm, runit=0.3}
\psframe[fillstyle=solid, fillcolor=lightgray](0,0)(2,2)
\psframe[fillstyle=solid, fillcolor=darkgray](2,2)(3,3) \psgrid[subgriddiv=1,gridlabels=0,linewidth=0.05mm](0,0)(14,14)
\psframe[linewidth=0.5mm](9,9)(14,14)
\psframe[linewidth=0.5mm](4,4)(8,8)
\psline(0.75,3.25)(0.25,3.75)\psline(0.25,3.25)(0.75,3.75)
\psline(1.75,2.25)(1.25,2.75)\psline(1.25,2.25)(1.75,2.75)
\psline(2.75,1.25)(2.25,1.75)\psline(2.25,1.25)(2.75,1.75)
\psline(3.25,0.25)(3.75,0.75)\psline(3.75,0.25)(3.25,0.75)
\psline(4.25,4.25)(4.75,4.75)\psline(4.75,4.25)(4.25,4.75)
\psline(5.25,5.25)(5.75,5.75)\psline(5.75,5.25)(5.25,5.75)
\psline(6.25,6.25)(6.75,6.75)\psline(6.75,6.25)(6.25,6.75)
\psline(7.25,7.25)(7.75,7.75)\psline(7.75,7.25)(7.25,7.75)
\psline(8.25,8.25)(8.75,8.75)\psline(8.75,8.25)(8.25,8.75)
\psline(9.25,9.25)(9.75,9.75)\psline(9.75,9.25)(9.25,9.75)
\psline(10.25,10.25)(10.75,10.75)\psline(10.75,10.25)(10.25,10.75)
\psline(11.25,11.25)(11.75,11.75)\psline(11.75,11.25)(11.25,11.75)
\psline(12.25,12.25)(12.75,12.75)\psline(12.75,12.25)(12.25,12.75)
\psline(13.25,13.25)(13.75,13.75)\psline(13.75,13.25)(13.25,13.75)
\pscircle(0.5,12.5){0.3}
\pscircle(1.5,8.5){0.3}
\pscircle(2.5,3.5){0.3}
\pscircle(3.5,7.5){0.3}
\pscircle(4.5,13.5){0.3}
\pscircle(5.5,2.5){0.3}
\pscircle(6.5,4.5){0.3}
\pscircle(7.5,5.5){0.3}
\pscircle(8.5,1.5){0.3}
\pscircle(9.5,0.5){0.3}
\pscircle(10.5,6.5){0.3}
\pscircle(11.5,9.5){0.3}
\pscircle(12.5,10.5){0.3}
\pscircle(13.5,11.5){0.3}
\psdots*[dotsize=.5](0,0)(1,4)(2,2)(3,3)(4,1)
\psdots*[dotsize=.5](5,5)(6,6)(7,7)(8,8)(9,9)(10,10)(11,11)(12,12)(13,13)
\end{pspicture} }

\subfigure[$y_3$]{ \begin{pspicture}(0,0)(4.2,4.2) \psset{xunit=0.3cm,yunit=0.3cm, runit=0.3} \psframe[fillstyle=solid, fillcolor=lightgray](0,0)(3,1)
\psframe[fillstyle=solid, fillcolor=darkgray](3,1)(4,2) \psgrid[subgriddiv=1,gridlabels=0,linewidth=0.05mm](0,0)(14,14)
\psframe[linewidth=0.5mm](9,9)(14,14) \psframe[linewidth=0.5mm](4,4)(8,8)
\psline(0.75,3.25)(0.25,3.75)\psline(0.25,3.25)(0.75,3.75)\psline(1.75,2.25)(1.25,2.75)\psline(1.25,2.25)(1.75,2.75)
\psline(2.75,1.25)(2.25,1.75)\psline(2.25,1.25)(2.75,1.75)\psline(3.25,0.25)(3.75,0.75)\psline(3.75,0.25)(3.25,0.75)
\psline(4.25,4.25)(4.75,4.75)\psline(4.75,4.25)(4.25,4.75)\psline(5.25,5.25)(5.75,5.75)\psline(5.75,5.25)(5.25,5.75)
\psline(6.25,6.25)(6.75,6.75)\psline(6.75,6.25)(6.25,6.75)\psline(7.25,7.25)(7.75,7.75)\psline(7.75,7.25)(7.25,7.75)
\psline(8.25,8.25)(8.75,8.75)\psline(8.75,8.25)(8.25,8.75)\psline(9.25,9.25)(9.75,9.75)\psline(9.75,9.25)(9.25,9.75)
\psline(10.25,10.25)(10.75,10.75)\psline(10.75,10.25)(10.25,10.75)
\psline(11.25,11.25)(11.75,11.75)\psline(11.75,11.25)(11.25,11.75)
\psline(12.25,12.25)(12.75,12.75)\psline(12.75,12.25)(12.25,12.75)
\psline(13.25,13.25)(13.75,13.75)\psline(13.75,13.25)(13.25,13.75)
\pscircle(0.5,12.5){0.3}\pscircle(1.5,8.5){0.3}\pscircle(2.5,3.5){0.3}\pscircle(3.5,7.5){0.3}\pscircle(4.5,13.5){0.3}\pscircle(5.5,2.5){0.3}\pscircle(6.5,4.5){0.3}\pscircle(7.5,5.5){0.3}\pscircle(8.5,1.5){0.3}\pscircle(9.5,0.5){0.3}\pscircle(10.5,6.5){0.3}\pscircle(11.5,9.5){0.3}\pscircle(12.5,10.5){0.3}\pscircle(13.5,11.5){0.3}
\psdots*[dotsize=.5](0,0)(1,4)(2,3)(3,1)(4,2)
\psdots*[dotsize=.5](5,5)(6,6)(7,7)(8,8)(9,9)(10,10)(11,11)(12,12)(13,13)
\end{pspicture}

}\subfigure[$x_1$]{ \begin{pspicture}(0,0)(4.2,4.2) \psset{xunit=0.3cm,yunit=0.3cm, runit=0.3} \psframe[fillstyle=solid, fillcolor=darkgray](0,0)(1,3)
\psframe[fillstyle=solid, fillcolor=lightgray](3,1)(4,2) \psgrid[subgriddiv=1,gridlabels=0,linewidth=0.05mm](0,0)(14,14)
\psframe[linewidth=0.5mm](9,9)(14,14) \psframe[linewidth=0.5mm](4,4)(8,8)
\psline(0.75,3.25)(0.25,3.75)\psline(0.25,3.25)(0.75,3.75)\psline(1.75,2.25)(1.25,2.75)\psline(1.25,2.25)(1.75,2.75)
\psline(2.75,1.25)(2.25,1.75)\psline(2.25,1.25)(2.75,1.75)\psline(3.25,0.25)(3.75,0.75)\psline(3.75,0.25)(3.25,0.75)
\psline(4.25,4.25)(4.75,4.75)\psline(4.75,4.25)(4.25,4.75)\psline(5.25,5.25)(5.75,5.75)\psline(5.75,5.25)(5.25,5.75)
\psline(6.25,6.25)(6.75,6.75)\psline(6.75,6.25)(6.25,6.75)\psline(7.25,7.25)(7.75,7.75)\psline(7.75,7.25)(7.25,7.75)
\psline(8.25,8.25)(8.75,8.75)\psline(8.75,8.25)(8.25,8.75)\psline(9.25,9.25)(9.75,9.75)\psline(9.75,9.25)(9.25,9.75)
\psline(10.25,10.25)(10.75,10.75)\psline(10.75,10.25)(10.25,10.75)
\psline(11.25,11.25)(11.75,11.75)\psline(11.75,11.25)(11.25,11.75)
\psline(12.25,12.25)(12.75,12.75)\psline(12.75,12.25)(12.25,12.75)
\psline(13.25,13.25)(13.75,13.75)\psline(13.75,13.25)(13.25,13.75)
\pscircle(0.5,12.5){0.3}\pscircle(1.5,8.5){0.3}\pscircle(2.5,3.5){0.3}\pscircle(3.5,7.5){0.3}\pscircle(4.5,13.5){0.3}\pscircle(5.5,2.5){0.3}\pscircle(6.5,4.5){0.3}\pscircle(7.5,5.5){0.3}\pscircle(8.5,1.5){0.3}\pscircle(9.5,0.5){0.3}\pscircle(10.5,6.5){0.3}\pscircle(11.5,9.5){0.3}\pscircle(12.5,10.5){0.3}\pscircle(13.5,11.5){0.3}
\psdots*[dotsize=.5](0,3)(1,0)(2,4)(3,2)(4,1)
\psdots*[dotsize=.5](5,5)(6,6)(7,7)(8,8)(9,9)(10,10)(11,11)(12,12)(13,13)
\end{pspicture} }\subfigure[$x_2$]{ \begin{pspicture}(0,0)(4.2,4.2) \psset{xunit=0.3cm,yunit=0.3cm, runit=0.3} \psframe[fillstyle=solid, fillcolor=darkgray](0,0)(2,2)
\psframe[fillstyle=solid, fillcolor=lightgray](3,1)(4,3) \psgrid[subgriddiv=1,gridlabels=0,linewidth=0.05mm](0,0)(14,14)
\psframe[linewidth=0.5mm](9,9)(14,14) \psframe[linewidth=0.5mm](4,4)(8,8)
\psline(0.75,3.25)(0.25,3.75)\psline(0.25,3.25)(0.75,3.75)\psline(1.75,2.25)(1.25,2.75)\psline(1.25,2.25)(1.75,2.75)
\psline(2.75,1.25)(2.25,1.75)\psline(2.25,1.25)(2.75,1.75)\psline(3.25,0.25)(3.75,0.75)\psline(3.75,0.25)(3.25,0.75)
\psline(4.25,4.25)(4.75,4.75)\psline(4.75,4.25)(4.25,4.75)\psline(5.25,5.25)(5.75,5.75)\psline(5.75,5.25)(5.25,5.75)
\psline(6.25,6.25)(6.75,6.75)\psline(6.75,6.25)(6.25,6.75)\psline(7.25,7.25)(7.75,7.75)\psline(7.75,7.25)(7.25,7.75)
\psline(8.25,8.25)(8.75,8.75)\psline(8.75,8.25)(8.25,8.75)\psline(9.25,9.25)(9.75,9.75)\psline(9.75,9.25)(9.25,9.75)
\psline(10.25,10.25)(10.75,10.75)\psline(10.75,10.25)(10.25,10.75)
\psline(11.25,11.25)(11.75,11.75)\psline(11.75,11.25)(11.25,11.75)
\psline(12.25,12.25)(12.75,12.75)\psline(12.75,12.25)(12.25,12.75)
\psline(13.25,13.25)(13.75,13.75)\psline(13.75,13.25)(13.25,13.75)
\pscircle(0.5,12.5){0.3}\pscircle(1.5,8.5){0.3}\pscircle(2.5,3.5){0.3}\pscircle(3.5,7.5){0.3}\pscircle(4.5,13.5){0.3}\pscircle(5.5,2.5){0.3}\pscircle(6.5,4.5){0.3}\pscircle(7.5,5.5){0.3}\pscircle(8.5,1.5){0.3}\pscircle(9.5,0.5){0.3}\pscircle(10.5,6.5){0.3}\pscircle(11.5,9.5){0.3}\pscircle(12.5,10.5){0.3}\pscircle(13.5,11.5){0.3}
\psdots*[dotsize=.5](0,2)(1,4)(2,0)(3,3)(4,1)
\psdots*[dotsize=.5](5,5)(6,6)(7,7)(8,8)(9,9)(10,10)(11,11)(12,12)(13,13)
\end{pspicture} }

\subfigure[$x_3$]{ \begin{pspicture}(0,0)(4.2,4.2) \psset{xunit=0.3cm,yunit=0.3cm, runit=0.3} \psframe[fillstyle=solid, fillcolor=darkgray](0,0)(3,1)
\psframe[fillstyle=solid, fillcolor=lightgray](2,2)(4,3)
\psframe[fillstyle=solid, fillcolor=lightgray](1,3)(2,4) \psgrid[subgriddiv=1,gridlabels=0,linewidth=0.05mm](0,0)(14,14)
\psframe[linewidth=0.5mm](9,9)(14,14) \psframe[linewidth=0.5mm](4,4)(8,8)
\psline(0.75,3.25)(0.25,3.75)\psline(0.25,3.25)(0.75,3.75)\psline(1.75,2.25)(1.25,2.75)\psline(1.25,2.25)(1.75,2.75)
\psline(2.75,1.25)(2.25,1.75)\psline(2.25,1.25)(2.75,1.75)\psline(3.25,0.25)(3.75,0.75)\psline(3.75,0.25)(3.25,0.75)
\psline(4.25,4.25)(4.75,4.75)\psline(4.75,4.25)(4.25,4.75)\psline(5.25,5.25)(5.75,5.75)\psline(5.75,5.25)(5.25,5.75)
\psline(6.25,6.25)(6.75,6.75)\psline(6.75,6.25)(6.25,6.75)\psline(7.25,7.25)(7.75,7.75)\psline(7.75,7.25)(7.25,7.75)
\psline(8.25,8.25)(8.75,8.75)\psline(8.75,8.25)(8.25,8.75)\psline(9.25,9.25)(9.75,9.75)\psline(9.75,9.25)(9.25,9.75)
\psline(10.25,10.25)(10.75,10.75)\psline(10.75,10.25)(10.25,10.75)
\psline(11.25,11.25)(11.75,11.75)\psline(11.75,11.25)(11.25,11.75)
\psline(12.25,12.25)(12.75,12.75)\psline(12.75,12.25)(12.25,12.75)
\psline(13.25,13.25)(13.75,13.75)\psline(13.75,13.25)(13.25,13.75)
\pscircle(0.5,12.5){0.3}\pscircle(1.5,8.5){0.3}\pscircle(2.5,3.5){0.3}\pscircle(3.5,7.5){0.3}\pscircle(4.5,13.5){0.3}\pscircle(5.5,2.5){0.3}\pscircle(6.5,4.5){0.3}\pscircle(7.5,5.5){0.3}\pscircle(8.5,1.5){0.3}\pscircle(9.5,0.5){0.3}\pscircle(10.5,6.5){0.3}\pscircle(11.5,9.5){0.3}\pscircle(12.5,10.5){0.3}\pscircle(13.5,11.5){0.3}
\psdots*[dotsize=.5](0,1)(1,4)(2,3)(3,0)(4,2)
\psdots*[dotsize=.5](5,5)(6,6)(7,7)(8,8)(9,9)(10,10)(11,11)(12,12)(13,13)
\end{pspicture}

}\subfigure[$y_4$]{ \begin{pspicture}(0,0)(4.2,4.2) \psset{xunit=0.3cm,yunit=0.3cm, runit=0.3} \psframe[fillstyle=solid, fillcolor=darkgray](3,1)(4,2)
\psframe[fillstyle=solid, fillcolor=lightgray](3,1)(1,0) \psframe[fillstyle=solid, fillcolor=lightgray](2,4)(3,14) \psgrid[subgriddiv=1,gridlabels=0,linewidth=0.05mm](0,0)(14,14)
\psframe[linewidth=0.5mm](9,9)(14,14) \psframe[linewidth=0.5mm](4,4)(8,8)
\psline(0.75,3.25)(0.25,3.75)\psline(0.25,3.25)(0.75,3.75)\psline(1.75,2.25)(1.25,2.75)\psline(1.25,2.25)(1.75,2.75)
\psline(2.75,1.25)(2.25,1.75)\psline(2.25,1.25)(2.75,1.75)\psline(3.25,0.25)(3.75,0.75)\psline(3.75,0.25)(3.25,0.75)
\psline(4.25,4.25)(4.75,4.75)\psline(4.75,4.25)(4.25,4.75)\psline(5.25,5.25)(5.75,5.75)\psline(5.75,5.25)(5.25,5.75)
\psline(6.25,6.25)(6.75,6.75)\psline(6.75,6.25)(6.25,6.75)\psline(7.25,7.25)(7.75,7.75)\psline(7.75,7.25)(7.25,7.75)
\psline(8.25,8.25)(8.75,8.75)\psline(8.75,8.25)(8.25,8.75)\psline(9.25,9.25)(9.75,9.75)\psline(9.75,9.25)(9.25,9.75)
\psline(10.25,10.25)(10.75,10.75)\psline(10.75,10.25)(10.25,10.75)
\psline(11.25,11.25)(11.75,11.75)\psline(11.75,11.25)(11.25,11.75)
\psline(12.25,12.25)(12.75,12.75)\psline(12.75,12.25)(12.25,12.75)
\psline(13.25,13.25)(13.75,13.75)\psline(13.75,13.25)(13.25,13.75)
\pscircle(0.5,12.5){0.3}\pscircle(1.5,8.5){0.3}\pscircle(2.5,3.5){0.3}\pscircle(3.5,7.5){0.3}\pscircle(4.5,13.5){0.3}\pscircle(5.5,2.5){0.3}\pscircle(6.5,4.5){0.3}\pscircle(7.5,5.5){0.3}\pscircle(8.5,1.5){0.3}\pscircle(9.5,0.5){0.3}\pscircle(10.5,6.5){0.3}\pscircle(11.5,9.5){0.3}\pscircle(12.5,10.5){0.3}\pscircle(13.5,11.5){0.3}
\psdots*[dotsize=.5](0,3)(1,0)(2,4)(3,1)(4,2)
\psdots*[dotsize=.5](5,5)(6,6)(7,7)(8,8)(9,9)(10,10)(11,11)(12,12)(13,13)
\end{pspicture} }\subfigure[$y_5$]{ \begin{pspicture}(0,0)(4.2,4.2) \psset{xunit=0.3cm,yunit=0.3cm, runit=0.3} \psframe[fillstyle=solid, fillcolor=darkgray](3,1)(4,3)
\psframe[fillstyle=solid, fillcolor=lightgray](2,0)(3,1) \psgrid[subgriddiv=1,gridlabels=0,linewidth=0.05mm](0,0)(14,14)
\psframe[linewidth=0.5mm](9,9)(14,14) \psframe[linewidth=0.5mm](4,4)(8,8)
\psline(0.75,3.25)(0.25,3.75)\psline(0.25,3.25)(0.75,3.75)\psline(1.75,2.25)(1.25,2.75)\psline(1.25,2.25)(1.75,2.75)
\psline(2.75,1.25)(2.25,1.75)\psline(2.25,1.25)(2.75,1.75)\psline(3.25,0.25)(3.75,0.75)\psline(3.75,0.25)(3.25,0.75)
\psline(4.25,4.25)(4.75,4.75)\psline(4.75,4.25)(4.25,4.75)\psline(5.25,5.25)(5.75,5.75)\psline(5.75,5.25)(5.25,5.75)
\psline(6.25,6.25)(6.75,6.75)\psline(6.75,6.25)(6.25,6.75)\psline(7.25,7.25)(7.75,7.75)\psline(7.75,7.25)(7.25,7.75)
\psline(8.25,8.25)(8.75,8.75)\psline(8.75,8.25)(8.25,8.75)\psline(9.25,9.25)(9.75,9.75)\psline(9.75,9.25)(9.25,9.75)
\psline(10.25,10.25)(10.75,10.75)\psline(10.75,10.25)(10.25,10.75)
\psline(11.25,11.25)(11.75,11.75)\psline(11.75,11.25)(11.25,11.75)
\psline(12.25,12.25)(12.75,12.75)\psline(12.75,12.25)(12.25,12.75)
\psline(13.25,13.25)(13.75,13.75)\psline(13.75,13.25)(13.25,13.75)
\pscircle(0.5,12.5){0.3}\pscircle(1.5,8.5){0.3}\pscircle(2.5,3.5){0.3}\pscircle(3.5,7.5){0.3}\pscircle(4.5,13.5){0.3}\pscircle(5.5,2.5){0.3}\pscircle(6.5,4.5){0.3}\pscircle(7.5,5.5){0.3}\pscircle(8.5,1.5){0.3}\pscircle(9.5,0.5){0.3}\pscircle(10.5,6.5){0.3}\pscircle(11.5,9.5){0.3}\pscircle(12.5,10.5){0.3}\pscircle(13.5,11.5){0.3}
\psdots*[dotsize=.5](0,2)(1,4)(2,0)(3,1)(4,3)
\psdots*[dotsize=.5](5,5)(6,6)(7,7)(8,8)(9,9)(10,10)(11,11)(12,12)(13,13)
\end{pspicture} }

\subfigure[$y_6$]{ \begin{pspicture}(0,0)(4.2,4.2) \psset{xunit=0.3cm,yunit=0.3cm, runit=0.3} \psframe[fillstyle=solid, fillcolor=darkgray](1,3)(2,4)
\psframe[fillstyle=solid, fillcolor=lightgray](0,1)(1,3) \psframe[fillstyle=solid, fillcolor=lightgray](2,4)(3,14)  \psgrid[subgriddiv=1,gridlabels=0,linewidth=0.05mm](0,0)(14,14)
\psframe[linewidth=0.5mm](9,9)(14,14) \psframe[linewidth=0.5mm](4,4)(8,8)
\psline(0.75,3.25)(0.25,3.75)\psline(0.25,3.25)(0.75,3.75)\psline(1.75,2.25)(1.25,2.75)\psline(1.25,2.25)(1.75,2.75)
\psline(2.75,1.25)(2.25,1.75)\psline(2.25,1.25)(2.75,1.75)\psline(3.25,0.25)(3.75,0.75)\psline(3.75,0.25)(3.25,0.75)
\psline(4.25,4.25)(4.75,4.75)\psline(4.75,4.25)(4.25,4.75)\psline(5.25,5.25)(5.75,5.75)\psline(5.75,5.25)(5.25,5.75)
\psline(6.25,6.25)(6.75,6.75)\psline(6.75,6.25)(6.25,6.75)\psline(7.25,7.25)(7.75,7.75)\psline(7.75,7.25)(7.25,7.75)
\psline(8.25,8.25)(8.75,8.75)\psline(8.75,8.25)(8.25,8.75)\psline(9.25,9.25)(9.75,9.75)\psline(9.75,9.25)(9.25,9.75)
\psline(10.25,10.25)(10.75,10.75)\psline(10.75,10.25)(10.25,10.75)
\psline(11.25,11.25)(11.75,11.75)\psline(11.75,11.25)(11.25,11.75)
\psline(12.25,12.25)(12.75,12.75)\psline(12.75,12.25)(12.25,12.75)
\psline(13.25,13.25)(13.75,13.75)\psline(13.75,13.25)(13.25,13.75)
\pscircle(0.5,12.5){0.3}\pscircle(1.5,8.5){0.3}\pscircle(2.5,3.5){0.3}\pscircle(3.5,7.5){0.3}\pscircle(4.5,13.5){0.3}\pscircle(5.5,2.5){0.3}\pscircle(6.5,4.5){0.3}\pscircle(7.5,5.5){0.3}\pscircle(8.5,1.5){0.3}\pscircle(9.5,0.5){0.3}\pscircle(10.5,6.5){0.3}\pscircle(11.5,9.5){0.3}\pscircle(12.5,10.5){0.3}\pscircle(13.5,11.5){0.3}
\psdots*[dotsize=.5](0,1)(1,3)(2,4)(3,0)(4,2)
\psdots*[dotsize=.5](5,5)(6,6)(7,7)(8,8)(9,9)(10,10)(11,11)(12,12)(13,13)
\end{pspicture}

}\subfigure[$y_7$]{ \begin{pspicture}(0,0)(4.2,4.2) \psset{xunit=0.3cm,yunit=0.3cm, runit=0.3} \psframe[fillstyle=solid, fillcolor=darkgray](2,2)(4,3)
\psframe[fillstyle=solid, fillcolor=lightgray](0,1)(2,2) \psgrid[subgriddiv=1,gridlabels=0,linewidth=0.05mm](0,0)(14,14)
\psframe[linewidth=0.5mm](9,9)(14,14) \psframe[linewidth=0.5mm](4,4)(8,8)
\psline(0.75,3.25)(0.25,3.75)\psline(0.25,3.25)(0.75,3.75)\psline(1.75,2.25)(1.25,2.75)\psline(1.25,2.25)(1.75,2.75)
\psline(2.75,1.25)(2.25,1.75)\psline(2.25,1.25)(2.75,1.75)\psline(3.25,0.25)(3.75,0.75)\psline(3.75,0.25)(3.25,0.75)
\psline(4.25,4.25)(4.75,4.75)\psline(4.75,4.25)(4.25,4.75)\psline(5.25,5.25)(5.75,5.75)\psline(5.75,5.25)(5.25,5.75)
\psline(6.25,6.25)(6.75,6.75)\psline(6.75,6.25)(6.25,6.75)\psline(7.25,7.25)(7.75,7.75)\psline(7.75,7.25)(7.25,7.75)
\psline(8.25,8.25)(8.75,8.75)\psline(8.75,8.25)(8.25,8.75)\psline(9.25,9.25)(9.75,9.75)\psline(9.75,9.25)(9.25,9.75)
\psline(10.25,10.25)(10.75,10.75)\psline(10.75,10.25)(10.25,10.75)
\psline(11.25,11.25)(11.75,11.75)\psline(11.75,11.25)(11.25,11.75)
\psline(12.25,12.25)(12.75,12.75)\psline(12.75,12.25)(12.25,12.75)
\psline(13.25,13.25)(13.75,13.75)\psline(13.75,13.25)(13.25,13.75)
\pscircle(0.5,12.5){0.3}\pscircle(1.5,8.5){0.3}\pscircle(2.5,3.5){0.3}\pscircle(3.5,7.5){0.3}\pscircle(4.5,13.5){0.3}\pscircle(5.5,2.5){0.3}\pscircle(6.5,4.5){0.3}\pscircle(7.5,5.5){0.3}\pscircle(8.5,1.5){0.3}\pscircle(9.5,0.5){0.3}\pscircle(10.5,6.5){0.3}\pscircle(11.5,9.5){0.3}\pscircle(12.5,10.5){0.3}\pscircle(13.5,11.5){0.3}
\psdots*[dotsize=.5](0,1)(1,4)(2,2)(3,0)(4,3)
\psdots*[dotsize=.5](5,5)(6,6)(7,7)(8,8)(9,9)(10,10)(11,11)(12,12)(13,13)
\end{pspicture} }\subfigure[$x_4$]{ \begin{pspicture}(0,0)(4.2,4.2) \psset{xunit=0.3cm,yunit=0.3cm, runit=0.3} \psframe[fillstyle=solid, fillcolor=darkgray](1,0)(3,1) \psframe[fillstyle=solid, fillcolor=darkgray](0,1)(1,3) \psgrid[subgriddiv=1,gridlabels=0,linewidth=0.05mm](0,0)(14,14)
\psframe[linewidth=0.5mm](9,9)(14,14) \psframe[linewidth=0.5mm](4,4)(8,8)
\psline(0.75,3.25)(0.25,3.75)\psline(0.25,3.25)(0.75,3.75)\psline(1.75,2.25)(1.25,2.75)\psline(1.25,2.25)(1.75,2.75)
\psline(2.75,1.25)(2.25,1.75)\psline(2.25,1.25)(2.75,1.75)\psline(3.25,0.25)(3.75,0.75)\psline(3.75,0.25)(3.25,0.75)
\psline(4.25,4.25)(4.75,4.75)\psline(4.75,4.25)(4.25,4.75)\psline(5.25,5.25)(5.75,5.75)\psline(5.75,5.25)(5.25,5.75)
\psline(6.25,6.25)(6.75,6.75)\psline(6.75,6.25)(6.25,6.75)\psline(7.25,7.25)(7.75,7.75)\psline(7.75,7.25)(7.25,7.75)
\psline(8.25,8.25)(8.75,8.75)\psline(8.75,8.25)(8.25,8.75)\psline(9.25,9.25)(9.75,9.75)\psline(9.75,9.25)(9.25,9.75)
\psline(10.25,10.25)(10.75,10.75)\psline(10.75,10.25)(10.25,10.75)
\psline(11.25,11.25)(11.75,11.75)\psline(11.75,11.25)(11.25,11.75)
\psline(12.25,12.25)(12.75,12.75)\psline(12.75,12.25)(12.25,12.75)
\psline(13.25,13.25)(13.75,13.75)\psline(13.75,13.25)(13.25,13.75)
\pscircle(0.5,12.5){0.3}\pscircle(1.5,8.5){0.3}\pscircle(2.5,3.5){0.3}\pscircle(3.5,7.5){0.3}\pscircle(4.5,13.5){0.3}\pscircle(5.5,2.5){0.3}\pscircle(6.5,4.5){0.3}\pscircle(7.5,5.5){0.3}\pscircle(8.5,1.5){0.3}\pscircle(9.5,0.5){0.3}\pscircle(10.5,6.5){0.3}\pscircle(11.5,9.5){0.3}\pscircle(12.5,10.5){0.3}\pscircle(13.5,11.5){0.3}
\psdots*[dotsize=.5](0,3)(1,1)(2,4)(3,0)(4,2)
\psdots*[dotsize=.5](5,5)(6,6)(7,7)(8,8)(9,9)(10,10)(11,11)(12,12)(13,13)
%\psframe[fillstyle=solid, fillcolor=lightgray](0,1)(1,3) \psframe[fillstyle=solid, fillcolor=lightgray](2,4)(3,14)
\end{pspicture} } \end{figure}

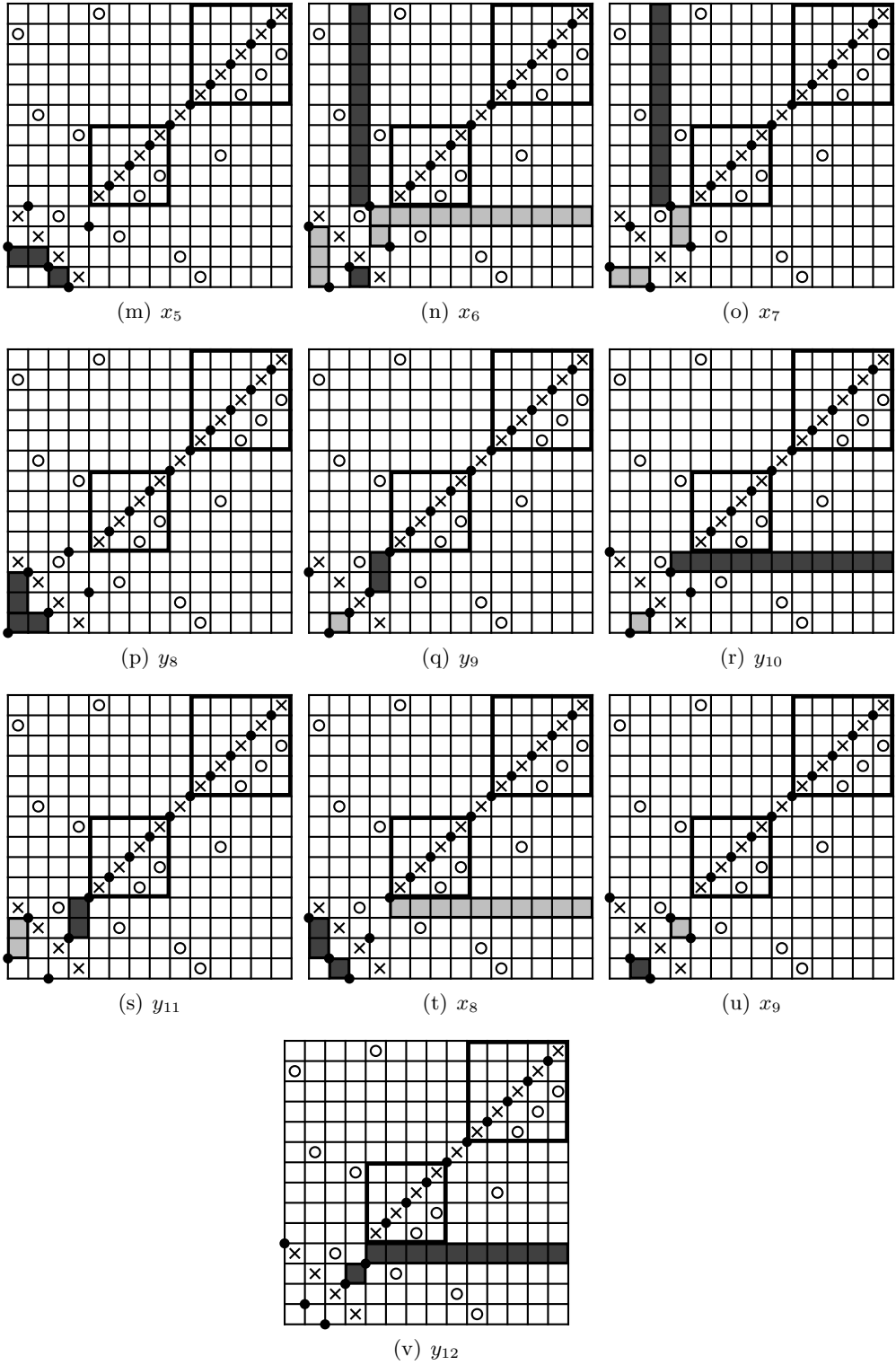
\begin{figure}[htbp] \centering \subfigure[$x_5$]{ \begin{pspicture}(0,0)(4.2,4.2) \psset{xunit=0.3cm,yunit=0.3cm, runit=0.3} \psframe[fillstyle=solid, fillcolor=darkgray](2,0)(3,1) \psframe[fillstyle=solid, fillcolor=darkgray](2,2)(0,1) \psgrid[subgriddiv=1,gridlabels=0,linewidth=0.05mm](0,0)(14,14)
\psframe[linewidth=0.5mm](9,9)(14,14) \psframe[linewidth=0.5mm](4,4)(8,8)
\psline(0.75,3.25)(0.25,3.75)\psline(0.25,3.25)(0.75,3.75)\psline(1.75,2.25)(1.25,2.75)\psline(1.25,2.25)(1.75,2.75)
\psline(2.75,1.25)(2.25,1.75)\psline(2.25,1.25)(2.75,1.75)\psline(3.25,0.25)(3.75,0.75)\psline(3.75,0.25)(3.25,0.75)
\psline(4.25,4.25)(4.75,4.75)\psline(4.75,4.25)(4.25,4.75)\psline(5.25,5.25)(5.75,5.75)\psline(5.75,5.25)(5.25,5.75)
\psline(6.25,6.25)(6.75,6.75)\psline(6.75,6.25)(6.25,6.75)\psline(7.25,7.25)(7.75,7.75)\psline(7.75,7.25)(7.25,7.75)
\psline(8.25,8.25)(8.75,8.75)\psline(8.75,8.25)(8.25,8.75)\psline(9.25,9.25)(9.75,9.75)\psline(9.75,9.25)(9.25,9.75)
\psline(10.25,10.25)(10.75,10.75)\psline(10.75,10.25)(10.25,10.75)
\psline(11.25,11.25)(11.75,11.75)\psline(11.75,11.25)(11.25,11.75)
\psline(12.25,12.25)(12.75,12.75)\psline(12.75,12.25)(12.25,12.75)
\psline(13.25,13.25)(13.75,13.75)\psline(13.75,13.25)(13.25,13.75)
\pscircle(0.5,12.5){0.3}\pscircle(1.5,8.5){0.3}\pscircle(2.5,3.5){0.3}\pscircle(3.5,7.5){0.3}\pscircle(4.5,13.5){0.3}\pscircle(5.5,2.5){0.3}\pscircle(6.5,4.5){0.3}\pscircle(7.5,5.5){0.3}\pscircle(8.5,1.5){0.3}\pscircle(9.5,0.5){0.3}\pscircle(10.5,6.5){0.3}\pscircle(11.5,9.5){0.3}\pscircle(12.5,10.5){0.3}\pscircle(13.5,11.5){0.3}
\psdots*[dotsize=.5](0,2)(1,4)(2,1)(3,0)(4,3)
\psdots*[dotsize=.5](5,5)(6,6)(7,7)(8,8)(9,9)(10,10)(11,11)(12,12)(13,13)
%\psframe[fillstyle=solid, fillcolor=lightgray](0,1)(2,2)
\end{pspicture}

}\subfigure[$x_6$]{ \begin{pspicture}(0,0)(4.2,4.2) \psset{xunit=0.3cm,yunit=0.3cm, runit=0.3} \psframe[fillstyle=solid, fillcolor=darkgray](2,4)(3,14) \psframe[fillstyle=solid, fillcolor=darkgray](2,0)(3,1)
\psframe[fillstyle=solid, fillcolor=lightgray](0,0)(1,3) \psframe[fillstyle=solid, fillcolor=lightgray](3,2)(4,4) \psframe[fillstyle=solid, fillcolor=lightgray](3,4)(14,3) \psgrid[subgriddiv=1,gridlabels=0,linewidth=0.05mm](0,0)(14,14)
\psframe[linewidth=0.5mm](9,9)(14,14) \psframe[linewidth=0.5mm](4,4)(8,8)
\psline(0.75,3.25)(0.25,3.75)\psline(0.25,3.25)(0.75,3.75)\psline(1.75,2.25)(1.25,2.75)\psline(1.25,2.25)(1.75,2.75)
\psline(2.75,1.25)(2.25,1.75)\psline(2.25,1.25)(2.75,1.75)\psline(3.25,0.25)(3.75,0.75)\psline(3.75,0.25)(3.25,0.75)
\psline(4.25,4.25)(4.75,4.75)\psline(4.75,4.25)(4.25,4.75)\psline(5.25,5.25)(5.75,5.75)\psline(5.75,5.25)(5.25,5.75)
\psline(6.25,6.25)(6.75,6.75)\psline(6.75,6.25)(6.25,6.75)\psline(7.25,7.25)(7.75,7.75)\psline(7.75,7.25)(7.25,7.75)
\psline(8.25,8.25)(8.75,8.75)\psline(8.75,8.25)(8.25,8.75)\psline(9.25,9.25)(9.75,9.75)\psline(9.75,9.25)(9.25,9.75)
\psline(10.25,10.25)(10.75,10.75)\psline(10.75,10.25)(10.25,10.75)
\psline(11.25,11.25)(11.75,11.75)\psline(11.75,11.25)(11.25,11.75)
\psline(12.25,12.25)(12.75,12.75)\psline(12.75,12.25)(12.25,12.75)
\psline(13.25,13.25)(13.75,13.75)\psline(13.75,13.25)(13.25,13.75)
\pscircle(0.5,12.5){0.3}\pscircle(1.5,8.5){0.3}\pscircle(2.5,3.5){0.3}\pscircle(3.5,7.5){0.3}\pscircle(4.5,13.5){0.3}\pscircle(5.5,2.5){0.3}\pscircle(6.5,4.5){0.3}\pscircle(7.5,5.5){0.3}\pscircle(8.5,1.5){0.3}\pscircle(9.5,0.5){0.3}\pscircle(10.5,6.5){0.3}\pscircle(11.5,9.5){0.3}\pscircle(12.5,10.5){0.3}\pscircle(13.5,11.5){0.3}
\psdots*[dotsize=.5](0,3)(1,0)(2,1)(3,4)(4,2)
\psdots*[dotsize=.5](5,5)(6,6)(7,7)(8,8)(9,9)(10,10)(11,11)(12,12)(13,13)
\end{pspicture} }\subfigure[$x_7$]{ \begin{pspicture}(0,0)(4.2,4.2) \psset{xunit=0.3cm,yunit=0.3cm, runit=0.3} \psframe[fillstyle=solid, fillcolor=darkgray](2,4)(3,14)
\psframe[fillstyle=solid, fillcolor=lightgray](0,0)(2,1) \psframe[fillstyle=solid, fillcolor=lightgray](3,2)(4,4) \psgrid[subgriddiv=1,gridlabels=0,linewidth=0.05mm](0,0)(14,14)
\psframe[linewidth=0.5mm](9,9)(14,14) \psframe[linewidth=0.5mm](4,4)(8,8)
\psline(0.75,3.25)(0.25,3.75)\psline(0.25,3.25)(0.75,3.75)\psline(1.75,2.25)(1.25,2.75)\psline(1.25,2.25)(1.75,2.75)
\psline(2.75,1.25)(2.25,1.75)\psline(2.25,1.25)(2.75,1.75)\psline(3.25,0.25)(3.75,0.75)\psline(3.75,0.25)(3.25,0.75)
\psline(4.25,4.25)(4.75,4.75)\psline(4.75,4.25)(4.25,4.75)\psline(5.25,5.25)(5.75,5.75)\psline(5.75,5.25)(5.25,5.75)
\psline(6.25,6.25)(6.75,6.75)\psline(6.75,6.25)(6.25,6.75)\psline(7.25,7.25)(7.75,7.75)\psline(7.75,7.25)(7.25,7.75)
\psline(8.25,8.25)(8.75,8.75)\psline(8.75,8.25)(8.25,8.75)\psline(9.25,9.25)(9.75,9.75)\psline(9.75,9.25)(9.25,9.75)
\psline(10.25,10.25)(10.75,10.75)\psline(10.75,10.25)(10.25,10.75)
\psline(11.25,11.25)(11.75,11.75)\psline(11.75,11.25)(11.25,11.75)
\psline(12.25,12.25)(12.75,12.75)\psline(12.75,12.25)(12.25,12.75)
\psline(13.25,13.25)(13.75,13.75)\psline(13.75,13.25)(13.25,13.75)
\pscircle(0.5,12.5){0.3}\pscircle(1.5,8.5){0.3}\pscircle(2.5,3.5){0.3}\pscircle(3.5,7.5){0.3}\pscircle(4.5,13.5){0.3}\pscircle(5.5,2.5){0.3}\pscircle(6.5,4.5){0.3}\pscircle(7.5,5.5){0.3}\pscircle(8.5,1.5){0.3}\pscircle(9.5,0.5){0.3}\pscircle(10.5,6.5){0.3}\pscircle(11.5,9.5){0.3}\pscircle(12.5,10.5){0.3}\pscircle(13.5,11.5){0.3}
\psdots*[dotsize=.5](0,1)(1,3)(2,0)(3,4)(4,2)
\psdots*[dotsize=.5](5,5)(6,6)(7,7)(8,8)(9,9)(10,10)(11,11)(12,12)(13,13)
\end{pspicture} }

\subfigure[$y_8$]{ \begin{pspicture}(0,0)(4.2,4.2) \psset{xunit=0.3cm,yunit=0.3cm, runit=0.3} \psframe[fillstyle=solid, fillcolor=darkgray](0,0)(1,3) \psframe[fillstyle=solid, fillcolor=darkgray](0,0)(2,1) \psgrid[subgriddiv=1,gridlabels=0,linewidth=0.05mm](0,0)(14,14)
\psframe[linewidth=0.5mm](9,9)(14,14) \psframe[linewidth=0.5mm](4,4)(8,8)
\psline(0.75,3.25)(0.25,3.75)\psline(0.25,3.25)(0.75,3.75)\psline(1.75,2.25)(1.25,2.75)\psline(1.25,2.25)(1.75,2.75)
\psline(2.75,1.25)(2.25,1.75)\psline(2.25,1.25)(2.75,1.75)\psline(3.25,0.25)(3.75,0.75)\psline(3.75,0.25)(3.25,0.75)
\psline(4.25,4.25)(4.75,4.75)\psline(4.75,4.25)(4.25,4.75)\psline(5.25,5.25)(5.75,5.75)\psline(5.75,5.25)(5.25,5.75)
\psline(6.25,6.25)(6.75,6.75)\psline(6.75,6.25)(6.25,6.75)\psline(7.25,7.25)(7.75,7.75)\psline(7.75,7.25)(7.25,7.75)
\psline(8.25,8.25)(8.75,8.75)\psline(8.75,8.25)(8.25,8.75)\psline(9.25,9.25)(9.75,9.75)\psline(9.75,9.25)(9.25,9.75)
\psline(10.25,10.25)(10.75,10.75)\psline(10.75,10.25)(10.25,10.75)
\psline(11.25,11.25)(11.75,11.75)\psline(11.75,11.25)(11.25,11.75)
\psline(12.25,12.25)(12.75,12.75)\psline(12.75,12.25)(12.25,12.75)
\psline(13.25,13.25)(13.75,13.75)\psline(13.75,13.25)(13.25,13.75)
\pscircle(0.5,12.5){0.3}\pscircle(1.5,8.5){0.3}\pscircle(2.5,3.5){0.3}\pscircle(3.5,7.5){0.3}\pscircle(4.5,13.5){0.3}\pscircle(5.5,2.5){0.3}\pscircle(6.5,4.5){0.3}\pscircle(7.5,5.5){0.3}\pscircle(8.5,1.5){0.3}\pscircle(9.5,0.5){0.3}\pscircle(10.5,6.5){0.3}\pscircle(11.5,9.5){0.3}\pscircle(12.5,10.5){0.3}\pscircle(13.5,11.5){0.3}
\psdots*[dotsize=.5](0,0)(1,3)(2,1)(3,4)(4,2)
\psdots*[dotsize=.5](5,5)(6,6)(7,7)(8,8)(9,9)(10,10)(11,11)(12,12)(13,13)
%\psframe[fillstyle=solid, fillcolor=lightgray](0,0)(1,3) \psframe[fillstyle=solid, fillcolor=lightgray](3,2)(4,4)
\end{pspicture}

}\subfigure[$y_9$]{ \begin{pspicture}(0,0)(4.2,4.2) \psset{xunit=0.3cm,yunit=0.3cm, runit=0.3} \psframe[fillstyle=solid, fillcolor=darkgray](3,2)(4,4)
\psframe[fillstyle=solid, fillcolor=lightgray](1,0)(2,1) \psgrid[subgriddiv=1,gridlabels=0,linewidth=0.05mm](0,0)(14,14)
\psframe[linewidth=0.5mm](9,9)(14,14) \psframe[linewidth=0.5mm](4,4)(8,8)
\psline(0.75,3.25)(0.25,3.75)\psline(0.25,3.25)(0.75,3.75)\psline(1.75,2.25)(1.25,2.75)\psline(1.25,2.25)(1.75,2.75)
\psline(2.75,1.25)(2.25,1.75)\psline(2.25,1.25)(2.75,1.75)\psline(3.25,0.25)(3.75,0.75)\psline(3.75,0.25)(3.25,0.75)
\psline(4.25,4.25)(4.75,4.75)\psline(4.75,4.25)(4.25,4.75)\psline(5.25,5.25)(5.75,5.75)\psline(5.75,5.25)(5.25,5.75)
\psline(6.25,6.25)(6.75,6.75)\psline(6.75,6.25)(6.25,6.75)\psline(7.25,7.25)(7.75,7.75)\psline(7.75,7.25)(7.25,7.75)
\psline(8.25,8.25)(8.75,8.75)\psline(8.75,8.25)(8.25,8.75)\psline(9.25,9.25)(9.75,9.75)\psline(9.75,9.25)(9.25,9.75)
\psline(10.25,10.25)(10.75,10.75)\psline(10.75,10.25)(10.25,10.75)
\psline(11.25,11.25)(11.75,11.75)\psline(11.75,11.25)(11.25,11.75)
\psline(12.25,12.25)(12.75,12.75)\psline(12.75,12.25)(12.25,12.75)
\psline(13.25,13.25)(13.75,13.75)\psline(13.75,13.25)(13.25,13.75)
\pscircle(0.5,12.5){0.3}\pscircle(1.5,8.5){0.3}\pscircle(2.5,3.5){0.3}\pscircle(3.5,7.5){0.3}\pscircle(4.5,13.5){0.3}\pscircle(5.5,2.5){0.3}\pscircle(6.5,4.5){0.3}\pscircle(7.5,5.5){0.3}\pscircle(8.5,1.5){0.3}\pscircle(9.5,0.5){0.3}\pscircle(10.5,6.5){0.3}\pscircle(11.5,9.5){0.3}\pscircle(12.5,10.5){0.3}\pscircle(13.5,11.5){0.3}
\psdots*[dotsize=.5](0,3)(1,0)(2,1)(3,2)(4,4)
\psdots*[dotsize=.5](5,5)(6,6)(7,7)(8,8)(9,9)(10,10)(11,11)(12,12)(13,13)
\end{pspicture} }\subfigure[$y_{10}$]{ \begin{pspicture}(0,0)(4.2,4.2) \psset{xunit=0.3cm,yunit=0.3cm, runit=0.3} \psframe[fillstyle=solid, fillcolor=darkgray](3,4)(14,3) \psframe[fillstyle=solid, fillcolor=lightgray](1,0)(2,1) \psgrid[subgriddiv=1,gridlabels=0,linewidth=0.05mm](0,0)(14,14)
\psframe[linewidth=0.5mm](9,9)(14,14) \psframe[linewidth=0.5mm](4,4)(8,8)
\psline(0.75,3.25)(0.25,3.75)\psline(0.25,3.25)(0.75,3.75)\psline(1.75,2.25)(1.25,2.75)\psline(1.25,2.25)(1.75,2.75)
\psline(2.75,1.25)(2.25,1.75)\psline(2.25,1.25)(2.75,1.75)\psline(3.25,0.25)(3.75,0.75)\psline(3.75,0.25)(3.25,0.75)
\psline(4.25,4.25)(4.75,4.75)\psline(4.75,4.25)(4.25,4.75)\psline(5.25,5.25)(5.75,5.75)\psline(5.75,5.25)(5.25,5.75)
\psline(6.25,6.25)(6.75,6.75)\psline(6.75,6.25)(6.25,6.75)\psline(7.25,7.25)(7.75,7.75)\psline(7.75,7.25)(7.25,7.75)
\psline(8.25,8.25)(8.75,8.75)\psline(8.75,8.25)(8.25,8.75)\psline(9.25,9.25)(9.75,9.75)\psline(9.75,9.25)(9.25,9.75)
\psline(10.25,10.25)(10.75,10.75)\psline(10.75,10.25)(10.25,10.75)
\psline(11.25,11.25)(11.75,11.75)\psline(11.75,11.25)(11.25,11.75)
\psline(12.25,12.25)(12.75,12.75)\psline(12.75,12.25)(12.25,12.75)
\psline(13.25,13.25)(13.75,13.75)\psline(13.75,13.25)(13.25,13.75)
\pscircle(0.5,12.5){0.3}\pscircle(1.5,8.5){0.3}\pscircle(2.5,3.5){0.3}\pscircle(3.5,7.5){0.3}\pscircle(4.5,13.5){0.3}\pscircle(5.5,2.5){0.3}\pscircle(6.5,4.5){0.3}\pscircle(7.5,5.5){0.3}\pscircle(8.5,1.5){0.3}\pscircle(9.5,0.5){0.3}\pscircle(10.5,6.5){0.3}\pscircle(11.5,9.5){0.3}\pscircle(12.5,10.5){0.3}\pscircle(13.5,11.5){0.3}
\psdots*[dotsize=.5](0,4)(1,0)(2,1)(3,3)(4,2)
\psdots*[dotsize=.5](5,5)(6,6)(7,7)(8,8)(9,9)(10,10)(11,11)(12,12)(13,13)
\end{pspicture} }

\subfigure[$y_{11}$]{ \begin{pspicture}(0,0)(4.2,4.2) \psset{xunit=0.3cm,yunit=0.3cm, runit=0.3} \psframe[fillstyle=solid, fillcolor=darkgray](3,2)(4,4)
\psframe[fillstyle=solid, fillcolor=lightgray](0,1)(1,3)  \psgrid[subgriddiv=1,gridlabels=0,linewidth=0.05mm](0,0)(14,14)
\psframe[linewidth=0.5mm](9,9)(14,14) \psframe[linewidth=0.5mm](4,4)(8,8)
\psline(0.75,3.25)(0.25,3.75)\psline(0.25,3.25)(0.75,3.75)\psline(1.75,2.25)(1.25,2.75)\psline(1.25,2.25)(1.75,2.75)
\psline(2.75,1.25)(2.25,1.75)\psline(2.25,1.25)(2.75,1.75)\psline(3.25,0.25)(3.75,0.75)\psline(3.75,0.25)(3.25,0.75)
\psline(4.25,4.25)(4.75,4.75)\psline(4.75,4.25)(4.25,4.75)\psline(5.25,5.25)(5.75,5.75)\psline(5.75,5.25)(5.25,5.75)
\psline(6.25,6.25)(6.75,6.75)\psline(6.75,6.25)(6.25,6.75)\psline(7.25,7.25)(7.75,7.75)\psline(7.75,7.25)(7.25,7.75)
\psline(8.25,8.25)(8.75,8.75)\psline(8.75,8.25)(8.25,8.75)\psline(9.25,9.25)(9.75,9.75)\psline(9.75,9.25)(9.25,9.75)
\psline(10.25,10.25)(10.75,10.75)\psline(10.75,10.25)(10.25,10.75)
\psline(11.25,11.25)(11.75,11.75)\psline(11.75,11.25)(11.25,11.75)
\psline(12.25,12.25)(12.75,12.75)\psline(12.75,12.25)(12.25,12.75)
\psline(13.25,13.25)(13.75,13.75)\psline(13.75,13.25)(13.25,13.75)
\pscircle(0.5,12.5){0.3}\pscircle(1.5,8.5){0.3}\pscircle(2.5,3.5){0.3}\pscircle(3.5,7.5){0.3}\pscircle(4.5,13.5){0.3}\pscircle(5.5,2.5){0.3}\pscircle(6.5,4.5){0.3}\pscircle(7.5,5.5){0.3}\pscircle(8.5,1.5){0.3}\pscircle(9.5,0.5){0.3}\pscircle(10.5,6.5){0.3}\pscircle(11.5,9.5){0.3}\pscircle(12.5,10.5){0.3}\pscircle(13.5,11.5){0.3}
\psdots*[dotsize=.5](0,1)(1,3)(2,0)(3,2)(4,4)
\psdots*[dotsize=.5](5,5)(6,6)(7,7)(8,8)(9,9)(10,10)(11,11)(12,12)(13,13)
\end{pspicture} }\subfigure[$x_8$]{ \begin{pspicture}(0,0)(4.2,4.2) \psset{xunit=0.3cm,yunit=0.3cm, runit=0.3} \psframe[fillstyle=solid, fillcolor=darkgray](1,3)(0,1) \psframe[fillstyle=solid, fillcolor=darkgray](2,1)(1,0)
\psframe[fillstyle=solid, fillcolor=lightgray](4,4)(14,3) \psgrid[subgriddiv=1,gridlabels=0,linewidth=0.05mm](0,0)(14,14)
\psframe[linewidth=0.5mm](9,9)(14,14) \psframe[linewidth=0.5mm](4,4)(8,8)
\psline(0.75,3.25)(0.25,3.75)\psline(0.25,3.25)(0.75,3.75)\psline(1.75,2.25)(1.25,2.75)\psline(1.25,2.25)(1.75,2.75)
\psline(2.75,1.25)(2.25,1.75)\psline(2.25,1.25)(2.75,1.75)\psline(3.25,0.25)(3.75,0.75)\psline(3.75,0.25)(3.25,0.75)
\psline(4.25,4.25)(4.75,4.75)\psline(4.75,4.25)(4.25,4.75)\psline(5.25,5.25)(5.75,5.75)\psline(5.75,5.25)(5.25,5.75)
\psline(6.25,6.25)(6.75,6.75)\psline(6.75,6.25)(6.25,6.75)\psline(7.25,7.25)(7.75,7.75)\psline(7.75,7.25)(7.25,7.75)
\psline(8.25,8.25)(8.75,8.75)\psline(8.75,8.25)(8.25,8.75)\psline(9.25,9.25)(9.75,9.75)\psline(9.75,9.25)(9.25,9.75)
\psline(10.25,10.25)(10.75,10.75)\psline(10.75,10.25)(10.25,10.75)
\psline(11.25,11.25)(11.75,11.75)\psline(11.75,11.25)(11.25,11.75)
\psline(12.25,12.25)(12.75,12.75)\psline(12.75,12.25)(12.25,12.75)
\psline(13.25,13.25)(13.75,13.75)\psline(13.75,13.25)(13.25,13.75)
\pscircle(0.5,12.5){0.3}\pscircle(1.5,8.5){0.3}\pscircle(2.5,3.5){0.3}\pscircle(3.5,7.5){0.3}\pscircle(4.5,13.5){0.3}\pscircle(5.5,2.5){0.3}\pscircle(6.5,4.5){0.3}\pscircle(7.5,5.5){0.3}\pscircle(8.5,1.5){0.3}\pscircle(9.5,0.5){0.3}\pscircle(10.5,6.5){0.3}\pscircle(11.5,9.5){0.3}\pscircle(12.5,10.5){0.3}\pscircle(13.5,11.5){0.3}
\psdots*[dotsize=.5](0,3)(1,1)(2,0)(3,2)(4,4)
\psdots*[dotsize=.5](5,5)(6,6)(7,7)(8,8)(9,9)(10,10)(11,11)(12,12)(13,13)
\end{pspicture} }\subfigure[$x_9$]{ \begin{pspicture}(0,0)(4.2,4.2) \psset{xunit=0.3cm,yunit=0.3cm, runit=0.3} \psframe[fillstyle=solid, fillcolor=darkgray](1,0)(2,1) \psframe[fillstyle=solid, fillcolor=lightgray](3,3)(4,2) \psgrid[subgriddiv=1,gridlabels=0,linewidth=0.05mm](0,0)(14,14)
\psframe[linewidth=0.5mm](9,9)(14,14) \psframe[linewidth=0.5mm](4,4)(8,8)
\psline(0.75,3.25)(0.25,3.75)\psline(0.25,3.25)(0.75,3.75)\psline(1.75,2.25)(1.25,2.75)\psline(1.25,2.25)(1.75,2.75)
\psline(2.75,1.25)(2.25,1.75)\psline(2.25,1.25)(2.75,1.75)\psline(3.25,0.25)(3.75,0.75)\psline(3.75,0.25)(3.25,0.75)
\psline(4.25,4.25)(4.75,4.75)\psline(4.75,4.25)(4.25,4.75)\psline(5.25,5.25)(5.75,5.75)\psline(5.75,5.25)(5.25,5.75)
\psline(6.25,6.25)(6.75,6.75)\psline(6.75,6.25)(6.25,6.75)\psline(7.25,7.25)(7.75,7.75)\psline(7.75,7.25)(7.25,7.75)
\psline(8.25,8.25)(8.75,8.75)\psline(8.75,8.25)(8.25,8.75)\psline(9.25,9.25)(9.75,9.75)\psline(9.75,9.25)(9.25,9.75)
\psline(10.25,10.25)(10.75,10.75)\psline(10.75,10.25)(10.25,10.75)
\psline(11.25,11.25)(11.75,11.75)\psline(11.75,11.25)(11.25,11.75)
\psline(12.25,12.25)(12.75,12.75)\psline(12.75,12.25)(12.25,12.75)
\psline(13.25,13.25)(13.75,13.75)\psline(13.75,13.25)(13.25,13.75)
\pscircle(0.5,12.5){0.3}\pscircle(1.5,8.5){0.3}\pscircle(2.5,3.5){0.3}\pscircle(3.5,7.5){0.3}\pscircle(4.5,13.5){0.3}\pscircle(5.5,2.5){0.3}\pscircle(6.5,4.5){0.3}\pscircle(7.5,5.5){0.3}\pscircle(8.5,1.5){0.3}\pscircle(9.5,0.5){0.3}\pscircle(10.5,6.5){0.3}\pscircle(11.5,9.5){0.3}\pscircle(12.5,10.5){0.3}\pscircle(13.5,11.5){0.3}
\psdots*[dotsize=.5](0,4)(1,1)(2,0)(3,3)(4,2)
\psdots*[dotsize=.5](5,5)(6,6)(7,7)(8,8)(9,9)(10,10)(11,11)(12,12)(13,13)
\end{pspicture} }

\subfigure[$y_{12}$]{ \begin{pspicture}(0,0)(4.2,4.2) \psset{xunit=0.3cm,yunit=0.3cm, runit=0.3} \psframe[fillstyle=solid, fillcolor=darkgray](3,3)(4,2) \psframe[fillstyle=solid, fillcolor=darkgray](4,4)(14,3) \psgrid[subgriddiv=1,gridlabels=0,linewidth=0.05mm](0,0)(14,14)
\psframe[linewidth=0.5mm](9,9)(14,14) \psframe[linewidth=0.5mm](4,4)(8,8)
\psline(0.75,3.25)(0.25,3.75)\psline(0.25,3.25)(0.75,3.75)\psline(1.75,2.25)(1.25,2.75)\psline(1.25,2.25)(1.75,2.75)
\psline(2.75,1.25)(2.25,1.75)\psline(2.25,1.25)(2.75,1.75)\psline(3.25,0.25)(3.75,0.75)\psline(3.75,0.25)(3.25,0.75)
\psline(4.25,4.25)(4.75,4.75)\psline(4.75,4.25)(4.25,4.75)\psline(5.25,5.25)(5.75,5.75)\psline(5.75,5.25)(5.25,5.75)
\psline(6.25,6.25)(6.75,6.75)\psline(6.75,6.25)(6.25,6.75)\psline(7.25,7.25)(7.75,7.75)\psline(7.75,7.25)(7.25,7.75)
\psline(8.25,8.25)(8.75,8.75)\psline(8.75,8.25)(8.25,8.75)\psline(9.25,9.25)(9.75,9.75)\psline(9.75,9.25)(9.25,9.75)
\psline(10.25,10.25)(10.75,10.75)\psline(10.75,10.25)(10.25,10.75)
\psline(11.25,11.25)(11.75,11.75)\psline(11.75,11.25)(11.25,11.75)
\psline(12.25,12.25)(12.75,12.75)\psline(12.75,12.25)(12.25,12.75)
\psline(13.25,13.25)(13.75,13.75)\psline(13.75,13.25)(13.25,13.75)
\pscircle(0.5,12.5){0.3}\pscircle(1.5,8.5){0.3}\pscircle(2.5,3.5){0.3}\pscircle(3.5,7.5){0.3}\pscircle(4.5,13.5){0.3}\pscircle(5.5,2.5){0.3}\pscircle(6.5,4.5){0.3}\pscircle(7.5,5.5){0.3}\pscircle(8.5,1.5){0.3}\pscircle(9.5,0.5){0.3}\pscircle(10.5,6.5){0.3}\pscircle(11.5,9.5){0.3}\pscircle(12.5,10.5){0.3}\pscircle(13.5,11.5){0.3}
\psdots*[dotsize=.5](0,4)(1,1)(2,0)(3,2)(4,3)
\psdots*[dotsize=.5](5,5)(6,6)(7,7)(8,8)(9,9)(10,10)(11,11)(12,12)(13,13)
\end{pspicture} }
\caption{Grid states for $G_2(a,b)$; here $(a,b)=(1,1)$.}
\end{figure}

\end{document}